\documentclass [oneside, 11pt, reqno] {amsart}

\usepackage[vmargin={2.8cm,3.3cm},hmargin=2.8cm,left=1.2in,right=1.2in]{geometry}
\usepackage [latin1] {inputenc}
\usepackage{amsmath}
\usepackage{amssymb}
\usepackage{amsfonts}
\usepackage{exscale}
\usepackage{graphicx}
\usepackage{setspace}
\usepackage{amsthm}
\usepackage{stmaryrd}

\makeatletter
\renewcommand\paragraph{\@startsection{paragraph}{4}{\z@}%
  {-3.25ex\@plus -1ex \@minus -.2ex}%
  {1.5ex \@plus .2ex}%
  {\normalfont\normalsize\itshape}}
\makeatother

\long\def\symbolfootnote[#1]#2{\begingroup%
\def\thefootnote{\fnsymbol{footnote}}\footnote[#1]{#2}\endgroup} 

\addtocontents{toc}{\protect\thispagestyle{empty}}

\newcounter{assumptions}

\newtheorem{ass}[assumptions]{Assumption}

\begin{document}

\newtheorem{lem}{Lemma}[section]
\newtheorem{cor}[lem]{Corollary}
\newtheorem{prop}[lem]{Proposition}
\newtheorem{thm}[lem]{Theorem}

\title[Non-asymptotic Error Bounds for Sequential MCMC]{Non-asymptotic Error Bounds for Sequential MCMC and Stability of Feynman-Kac Propagators}
\author{Nikolaus Schweizer}
\date{March 2012}

\address{Institute for Applied Mathematics, University of Bonn, Endenicher Allee 60, 53115 Bonn}

\email{nschweizer@iam.uni-bonn.de}
\thanks{Financial support of the German Research Foundation (DFG)  through the Hausdorff Center for Mathematics is gratefully acknowledged.}

\date{March 2012}

\subjclass[2000]{65C05, 60J10, 60B10, 47H20, 47D08}
\keywords{Markov Chain Monte Carlo, sequential Monte Carlo, importance sampling,
spectral gap, tempering, Feynman-Kac formula}

\begin{abstract}
We provide a generic way of deducing non-asymptotic error bounds for Sequential MCMC methods from suitable stability properties of Feynman-Kac propagators. We show how to derive this type of stability from mixing conditions for the MCMC dynamics, namely, spectral gaps and hyperboundedness,  and from upper bounds on the relative densities in the sequence of distributions. 
\end{abstract}

\maketitle

\section{Introduction}

Since the 1950s (Metropolis et al. \cite{MRRTT53}), Markov Chain Monte Carlo (MCMC) methods have become an increasingly popular tool for challenging numerical integration problems in a wide variety of fields ranging from chemical physics to financial econometrics (see, e.g. Liu \cite{Liu01}). The basic idea is to approximate the integral of a function $f$ with respect to a probability measure $\mu$ by simulating a Markov chain with stationary distribution $\mu$ and to calculate the ergodic average of $f$ evaluated at the positions visited by the Markov chain. By construction, MCMC methods only work well if the simulated Markov chain reaches equilibrium sufficiently quickly. Roughly speaking, this is the case when $\mu$ is essentially unimodal and it is not the case when $\mu$ is severely multimodal in the sense of being characterized by several well-separated modes. In the latter case, MCMC methods tend to get stuck in local modes for very long times and therefore approach their equilibrium $\mu$ only on time-scales well beyond those that can feasibly be simulated. This metastability phenomenon is a serious drawback of standard MCMC methods in many applications.\bigskip

Sequential MCMC methods (see, e.g., Del Moral, Doucet and Jasra \cite{DDJ05} and the references therein) are a class of algorithms which try to overcome this problem. The basic idea is to approximate the target distribution $\mu$ with a sequence of distributions $\mu_0,\ldots, \mu_n$ such that $\mu_n=\mu$ is the actual target distribution and such that $\mu_0$ is easy to sample from. The aim is to carry over the good mixing/sampling properties from $\mu_0$ to the target $\mu_n$ using interpolating distributions $\mu_k$ where $\mu_{k}$ and $\mu_{k+1}$ are sufficiently similar to allow for, e.g., efficient importance sampling. The algorithm constructs a system of $N$ particles which sequentially approximates the measures $\mu_0$ to $\mu_n$. The algorithm is initialized with $N$ independent samples from $\mu_0$ and then alternates two types of steps, Importance Sampling Resampling and MCMC: In the Importance Sampling Resampling steps, a cloud of particles approximating $\mu_k$ is transformed into a cloud of particles approximating $\mu_{k+1}$ by randomly duplicating and eliminating particles in a suitable way, depending on the relative density between $\mu_{k+1}$ and $\mu_{k}$. This step is similar to the selection step in models of population genetics where particles form the population and where the relative density takes the role of a fitness function guiding the number of off-spring a particle has. In the MCMC steps, particles move independently according to an MCMC dynamics for the current target distribution in order to adjust better to the changed environment. This step resembles the mutation step in models of population genetics. The algorithm is essentially the same as the particle filter of Gordon, Salmond and Smith \cite{GSS93} yet the application - sampling from a fixed target distribution instead of filtering with an exogenous sequence of distributions - is different. The error analysis of this paper focuses on a simple Sequential MCMC algorithm with Multinomial Resampling which is a basic instance of the class of algorithms introduced in Del Moral, Doucet and Jasra \cite{DDJ05}.\bigskip

This paper contributes to the literature on error bounds and stability analysis of Sequential MCMC. Specifically, we are interested in proving explicit non-asymptotic error bounds which are well-suited to the high-dimensional and non-compact state spaces prevalent in many applications of Monte Carlo Methods. Our two main results can be summarized as follows: Theorem \ref{thmBound} provides a generic way for deriving non-asymptotic bounds on the mean squared integration error from suitable stability properties of the Feynman-Kac propagator associated with the particle system. The bound of the theorem applies to the so-called unnormalized particle approximation but immediately implies a bound for the actual particle approximation provided by the algorithm, see Lemma \ref{EMbounds}. The bound in the theorem is sharp in the sense that the leading error term corresponds to the asymptotic variance of the unnormalized particle approximation found, e.g., in the central limit theorem of Del Moral and Miclo \cite{DM00}. Theorem \ref{propLpIt}  derives $L_p$-stability bounds for the Feynman-Kac propagator which can be applied in the framework of  Theorem \ref{thmBound}. Specifically, we derive stability in $L_p$ from sufficiently good mixing of the MCMC dynamics in the sense of a sufficiently large spectral gap and from uniform upper bounds on relative densities. $L_{2p}$-$L_p$-bounds, which we need in our error bound for Sequential MCMC, follow under an additional assumption of hyperboundedness for the MCMC dynamics.\bigskip 

Basically, our results are discrete-time, general state space analogues of the error bounds and stability results given, respectively, in Eberle and Marinelli \cite{EM08}, \cite{EM09} who analyze a related continuous-time particle system with simultaneous resampling and MCMC on a finite state space. The generalization to discrete time is important since this corresponds to the algorithms implemented in practice. Moreover, in the discrete time setting a uniform lower bound on relative densities is not needed to derive the results.\bigskip

There is by now a sizeable literature on asymptotic error bounds for Sequential MCMC and related particle systems beginning with the central limit theorems in Del Moral \cite{DM96}, Chopin \cite{C04}, Künsch \cite{K05} and Capp\'e, Moulines and Ryd\'en, \cite{CMR05}). See Del Moral \cite{DM05} for an overview and many results, and Douc and Moulines \cite{DM08} for a recent contribution. Non-asymptotic error bounds, i.e., error bounds for a fixed number of particles are comparatively less studied, see Del Moral and Miclo \cite{DM00}, Theorem 7.4.4 of Del Moral \cite{DM05}, C\'erou, Del Moral and Guyader \cite{CDG11} and the references in these for some results. These works rely on mixing and boundedness conditions which will usually not be satisfied on a non-compact state space. Namely, this concerns uniform upper and lower bounds on relative densities between $\mu_{k+1}$ and $\mu_{k}$ and  uniformly bounded relative densities between MCMC transition kernels from different starting points.  To see the limitation of such conditions consider the case where $\mu_k$ and $\mu_{k+1}$ are two different Gaussian distributions on $\mathbb{R}$. Then the relative density tends either to zero or to infinity in the tails. Similarly, the transition kernels of most commonly used MCMC dynamics will be almost singular for starting points which lie far apart.\bigskip

To see why upper bounds on relative densities are less restrictive than lower bounds in typical applications of Sequential MCMC, note that the most common way of choosing the distributions $\mu_k$ is by setting $\mu_k(\cdot) \sim \exp(-\beta_k H(\cdot))$ for an increasing sequence $\beta_k$ of (artificial) inverse temperature parameters, see, e.g., Neal \cite{N01}. Then the relative density from $\mu_k$ to $\mu_{k+1}$ is proportional to  $\exp(-(\beta_{k+1}-\beta_k) H(\cdot))$ which is bounded from above whenever the Hamiltonian $H$ is bounded from below which is the case in most applications.\bigskip

Recently, Whiteley \cite{W11} proved non-asymptotic error bounds that can be expected to be applicable to non-trivial models with non-compact state spaces. In Whiteley's setting, the usual global mixing condition is replaced by a family of minorization conditions on a small set and drift conditions outside the small set, applying results of Douc, Moulines and Rosenthal \cite{DMR04}. These results are complementary to ours since they rely on a rather different set of conditions. In addition, our results are more explicit allowing, e.g, to make rather precise predictions about the dimension dependence of the algorithm. Whiteley's results have a somewhat wider scope, applying, e.g., to unbounded integrands and settings with an initial bias.\bigskip

The literature on non-asymptotic error bounds - including this contribution - has largely focused on the case where the underlying MCMC dynamics mixes well globally for all the distributions $\mu_k$. This is, of course, unfortunate since one main motivation for applying a multilevel MCMC method such as Sequential MCMC is the presence of well-separated modes in the target distribution which renders the application of simple MCMC dynamics infeasible. An exception is Eberle and Marinelli \cite{EM09, EM08} who address in their continuous-time framework the case in which the state space is decomposed into several components that are not connected for all $\mu_k$. In a follow-up paper, Schweizer \cite{Schw12}), bounds for more general multimodal cases are derived from Theorem \ref{thmBound} presented here.\bigskip

The results contained in here are extracted from a more detailed presentation in the dissertation \cite{Schw11}. The paper is structured as follows: Section \ref{Preliminaries} introduces the setting and recalls a number of results from the literature, chiefly, an explicit expression for the variance of the reweighted particle approximation. Section \ref{secQuantCBG} shows how to derive error bounds from suitable stability properties of the Feynman-Kac propagator associated with the model. Section \ref{LPglobal} shows how to derive these stability properties from mixing conditions of the MCMC dynamics. Section \ref{GLOPOLOSO} gives more explicit statements of our error bounds in terms of the constants in Poincaré and Logarithmic-Sobolev inequalities. Finally, Section \ref{secDD} gives a brief discussion of the dependence of the algorithm's error on the dimension of the state space in light of our results.\bigskip

\section{Preliminaries}\label{Preliminaries}

\subsection{Notation} 

Let $(E,r)$ be a complete, separable metric space and let $\mathcal{B}(E)$ be the $\sigma$-algebra of Borel subsets of $E$. Denote by $M(E)$ the space of finite signed Borel measures on $E$. Let $M_1(E)\subset M(E)$ be the subset of all probability measures. Let $B(E)$ be the space of bounded, measurable, real-valued functions on $E$.\bigskip

For $\mu\in M(E)$ and $f\in B(E)$ define $\mu(f)$ by
\[
\mu(f)=\int_E f(x) \mu(dx)
\]
and $\text{Var}_\mu(f)$ by
\[
\text{Var}_\mu(f)=\mu(f^2)-\mu(f)^2.
\]

Let $(\widetilde{E},\widetilde{r})$ be another Polish space. Consider a kernel $K(x,A)$ with $K(x,\cdot)\in M(\widetilde{E})$ for $x \in E$ and $K(\cdot, A)\in B(E)$ for $A\in \mathcal{B}(\widetilde{E})$. We define for $\mu\in M(E)$ the measure $\mu K\in M(\widetilde{E})$ by
\[
\mu K(A)=\int_E K(x,A)\mu (dx) \;\;\;\; \forall A \in \mathcal{B}(\widetilde{E}). 
\]
For $f \in B(\widetilde{E})$ we denote by $K(f)\in B(E)$ the function given by
\[
K(f)(x)=K(x,f)= \int_E f(z) K(x,dz)\;\;\;\; \forall x\in E.
\]

\subsection{The Measure-Valued Model}\label{tmvmodel}
Consider a sequence of Polish spaces $(E_k,r_k)$ and a sequence of probability measures $(\mu_k)_{k=0}^n$, $\mu_k \in M_1(E_k)$. This is the sequence of measures we wish to approximate with the algorithm introduced in Section \ref{IPS}. The measures $\mu_k$ are related through
\[
\mu_k(f)= \frac{\mu_{k-1}(g_{k-1,k} K_k(f))}{\mu_{k-1}(g_{k-1,k})} \;\;\;\; \forall f \in B(E_k) 
\]
for positive functions $g_{k-1,k}\in B(E_{k-1})$ and transition kernels $K_k$ with $K_k(x,\cdot)\in M_1(E_{k})$ for $x \in E_{k-1}$ and $K_k(\cdot, A)\in B(E_{k-1})$ for $A\in \mathcal{B}(E_k)$. We define the probability distribution $\hat{\mu}_k \in M_1(E_{k-1})$ by 
\[
\hat{\mu}_k(f)= \frac{\mu_{k-1}(g_{k-1,k} f)}{\mu_{k-1}(g_{k-1,k})} \;\;\;\; \forall f \in B(E_{k-1}).
\]
This implies $\hat{\mu}_k(K_k(f))=\mu_k(f)$ for $f\in B(E_k)$.\bigskip

Next we introduce the Feynman-Kac propagator $q_{j,k}$ which will be the central object of our error analysis. Define the mapping $q_{k-1,k}:B(E_k)\rightarrow B(E_{k-1})$ by
\[
q_{k-1,k}(f)=\frac{g_{k-1,k}K_k(f)}{\mu_{k-1}(g_{k-1,k})}. 
\]
Observe that this implies
\[
\mu_{k}(f)=\mu_{k-1}(q_{k-1,k}(f)).
\]
Furthermore define for $0 \leq j < k \leq n$ the mapping $q_{j,k}:B(E_k)\rightarrow B(E_j)$ by
\[
q_{j,k}(f)=q_{j,j+1}(q_{j+1,j+2}(\ldots q_{k-1,k}(f)))
\]
and $q_{k,k}(f)=f$.
Note that for $f\in B(E_k)$ we have the relation 
\[
\mu_j(q_{j,k}(f))=\mu_k(f) \;\;\;\text{ for }0\leq j \leq  k \leq n
\]
and the semigroup property
\[
q_{j,l} (q_{l,k}(f)) = q_{j,k}(f) \;\;\;\text{ for }0 \leq j <l< k \leq n.
\]

\subsection{The Interacting Particle System}\label{IPS}

In the Sequential MCMC algorithm, we approximate the measures $(\mu_k)_k$ by simulating the interacting particle system introduced in the following. We start with $N$ independent samples $\xi_0=(\xi_0^1,\ldots,\xi_0^N)$ from $\mu_0$. The particle dynamics alternates two steps: Importance Sampling Resampling and Mutation: A vector of particles $\xi_{k-1}$ approximating $\mu_{k-1}$ is transformed into a vector $\hat{\xi}_k$ approximating $\hat{\mu}_k$ by drawing $N$ conditionally independent samples from the empirical distribution of $\xi_{k-1}$ weighted with the functions $g_{k-1,k}$. Afterwards, $\hat{\xi}_k$ is transformed into a vector $\xi_k$ approximating $\mu_k$ by moving the particles $\hat{\xi}_k^i$ independently with the transition kernel $K_k$.\bigskip 

We thus have two arrays of random variables $(\xi_k^j)_{0\leq k\leq n,1\leq j\leq N}$ and $(\hat{\xi}_k^j)_{1\leq k\leq n,1\leq j\leq N}$ where $\xi_k^j$ and $\hat{\xi}_{k+1}^j$ take values in $E_k$. Denote respectively by $\mathbb{P}[\cdot]$ and $\mathbb{E}[\cdot]$ probabilities and expectations taken with respect to the randomness in the particle system, i.e., with respect to the random variables $(\xi_k^j)_{k,j}$ and $(\hat{\xi}_k^j)_{k,j}$. Denote by $\mathcal{F}_k$ the $\sigma$-algebra generated by $\xi_0,\ldots \xi_k$ and $\hat{\xi}_1,\ldots \hat{\xi}_k$ and by $\hat{\mathcal{F}}_k$ the $\sigma$-algebra generated by $\xi_0,\ldots \xi_{k-1}$ and $\hat{\xi}_1,\ldots \hat{\xi}_k$. Denote by $\eta_k^N$ the empirical measure of $\xi_k$, i.e.,
\[
\eta_k^N=\frac1N \sum_{i=1}^N \delta_{\xi_k^i}.
\]
The algorithm proceeds as follows: 

\begin{itemize}
\item[(i)] Draw $\xi_0^1,\ldots,\xi_0^N$ independently from $\mu_0$.
\item[(ii)] For $k=1,\ldots,n$,
	\begin{itemize}
		\item[(a)] draw $\hat{\xi}_k=(\hat{\xi}_k^1,\ldots,\hat{\xi}_k^N)$ according to 
			\[
				\mathbb{P}[\hat{\xi}_k\in dx | \mathcal{F}_{k-1} ]= \prod_{j=1}^N \sum_{i=1}^N \frac{g_{k-1,k}(\xi_{k-1}^i)}{\sum_{l=1}^N 	
					g_{k-1,k}(\xi_{k-1}^l)}\delta_{\xi_{k-1}^i}(dx^j),
			\]
		\item[(b)] draw $\xi_k=(\xi_k^1,\ldots,\xi_k^N)$ according to
			\[
			\mathbb{P}[\xi_{k}\in dx | \hat{\mathcal{F}}_k ]= \prod_{j=1}^N K_k(\hat{\xi}_k^j,dx^j).
			\]
	\end{itemize}
\item[(iii)] Approximate $\mu_n(f)$ by
\[
\eta_n^N(f)=\frac1N \sum_{i=1}^N f(\xi_n^i).
\]
\end{itemize}

In the following we will study, how well $\eta_n^N$ approximates $\mu_n$. We are mainly interested in the case where the state spaces $E_k$ are identical and where $K_k$ encompasses many steps of an MCMC dynamics (e.g., Metropolis) with target $\mu_k$. In that case $g_{k-1,k}$ becomes an unnormalized relative density between $\mu_{k-1}$ and $\mu_k$ and we have $\hat{\mu}_k=\mu_k$. We will restrict attention to that setting from Section \ref{LPglobal} on, but up to then this additional generality comes basically for free. At that point we will also add assumptions ensuring that $\mu_{k-1}$ and $\mu_k$ are sufficiently similar (namely, Assumption \ref{ass1}).\bigskip

We end the preliminary observations with the following lemma:\bigskip

\begin{lem}\label{bederweta} We have for $f\in B(E_k)$ and $1\leq k \leq n$
\[
\mathbb{E}[\eta_k^N (f)| \mathcal{F}_{k-1}]=\frac{\eta_{k-1}^N(q_{k-1,k}(f))}{\eta_{k-1}^N(q_{k-1,k}(1))} 
\]
and for $1 \leq j \leq N$
\[
\mathbb{E}[f(\xi_k^j)| \mathcal{F}_{k-1}]=\frac{\eta_{k-1}^N(q_{k-1,k}(f))}{\eta_{k-1}^N(q_{k-1,k}(1))}. 
\]
\end{lem}

\begin{proof}[Proof]
Note that $\mathcal{F}_{k-1} \subseteq \hat{\mathcal{F}}_k \subseteq \mathcal{F}_k$. We can thus write
\begin{eqnarray}
\mathbb{E}[f(\xi_k^j)| \mathcal{F}_{k-1}]&=&\mathbb{E}[\mathbb{E}[f(\xi_k^j)|\hat{\mathcal{F}}_k ]| \mathcal{F}_{k-1}]\nonumber\\
&=&\mathbb{E}[K_k(\hat{\xi}_{k}^j,f)| \mathcal{F}_{k-1}]\nonumber\\
&=&\frac{\sum_{i=1}^N g_{k-1,k}(\xi_{k-1}^i)K_k(\xi_{k-1}^i,f)}{\sum_{l=1}^N g_{k-1,k}(\xi_{k-1}^l)}\nonumber\\
&=&\frac{\eta_{k-1}^N (g_{k-1,k}K_k(f))}{\eta_{k-1}^N(g_{k-1,k})}\nonumber\\
&=&\frac{\eta_{k-1}^N(q_{k-1,k}(f))}{\eta_{k-1}^N(q_{k-1,k}(1))}. \nonumber
\end{eqnarray}
This proves the claim for $\mathbb{E}[f(\xi_k^j)| \mathcal{F}_{k-1}]$ and immediately implies the claim for $\mathbb{E}[\eta_k^N (f)| \mathcal{F}_{k-1}]$.
\end{proof}

\subsection{Variances of Weighted Empirical Averages}\label{VariancesWEA}

We are interested in finding efficient upper bounds for the quantities
\[
\mathbb{E}[|\eta_n^N(f)-\mu_n(f)|^2]
\]
and
\[
\mathbb{E}[|\eta_n^N(f)-\mu_n(f)|].
\]
As is shown in Lemma \ref{EMbounds} below, these quantities can be controlled in terms of the approximation error of a weighted empirical measure $\nu_n^N(f)$ which is easier to handle. In this section, we introduce the measure  $\nu_n^N(f)$ and establish an explicit formula for its quadratic error
\[
\mathbb{E}[|\nu_n^N(f)-\mu_n(f)|^2].
\]

Define for $0\leq k\leq n$
\[
\nu_k^N(f)=\varphi_k \, \eta_k^N(f),
\]
where $\varphi_k$ is given by
\[
\varphi_k=\prod_{j=0}^{k-1} \eta_j^N( q_{j,j+1}(1) )
\]
for $1\leq k \leq n$ and $\varphi_0=1$. Note that typically $\nu_k^N$ is not a probability distribution and that $\varphi_k$ is $\mathcal{F}_{k-1}$-measurable. The factor $\varphi_k$ is chosen in a way that from Lemma \ref{bederweta} we have
\begin{equation}\label{erwetak}
\mathbb{E}[\nu_k^N (f)| \mathcal{F}_{k-1}]=\nu_{k-1}^N ( q_{k-1,k}(f) ).
\end{equation}
By the unbiasedness proved in Proposition \ref{propVarNu} below it also holds that
\[
\mathbb{E}[\varphi_k]=\mathbb{E}[\nu_k^N(1)]=\mathbb{E}[\mu_k(1)]=1.
\]
Furthermore, the following relation will prove useful:
\begin{equation}\label{phiketa}
\nu_{k+1}^N(1)=\varphi_{k+1}=\varphi_{k}\eta_k^N(q_{k,k+1}(1))=\nu_{k}^N(q_{k,k+1}(1)).
\end{equation}

The connection between the approximation errors of $\eta_n^N(f)$ and $\nu_n^N(f)$ is established in the following lemma:\bigskip

\begin{lem}\label{EMbounds}
For $f\in B(E_n)$ define $f_n=f-\mu_n(f)$  and denote by $\|\cdot\|_{\text{\textnormal{sup}},n}$ the supremum norm on $B(E_n)$. Then we have the bounds
\begin{equation}\label{EMbound1}
\mathbb{E}[(\eta_n^N(f)-\mu_n(f))^2]\leq 2\,\text{\textnormal{Var}}(\nu_n^N(f_n))+2\,\|f_n \|_{\text{\textnormal{sup}},n}^2\text{\textnormal{Var}}(\nu_n^N(1))
\end{equation}
and
\begin{eqnarray}\label{EMbound2}
&&\mathbb{E}[|\eta_n^N(f)-\mu_n(f)|]\\&\leq& \text{\textnormal{Var}}(\nu_n^N(f_n))^{\frac12}+\sqrt{2}\|f_n \|_{\text{\textnormal{sup}},n}\text{\textnormal{Var}}(\nu_n^N(1))+ \sqrt{2}\text{Var}(\nu_n^N(f_n))^{\frac12}\text{\textnormal{Var}}(\nu_n^N(1))^{\frac12}.\nonumber
\end{eqnarray}
\end{lem}

\begin{proof}[Proof of Lemma \ref{EMbounds}]
Observe that for $a,b\in \mathbb{R}$ the fact that $(a-2b)^2\geq 0$ implies
\[
a^2 \leq 2(a-b)^2+2 b^2.
\]
We can thus prove (\ref{EMbound1}) as follows:
\begin{eqnarray}
\mathbb{E}[\eta_n^N(f_n)^2]&\leq& 2 \mathbb{E}[(\eta_n^N(f_n)-\nu_n^N(f_n))^2]+2 \mathbb{E}[\nu_n^N(f_n)^2]\nonumber\\
&\leq& 2\|f_n\|_{\text{\textnormal{sup}},n}^2 \text{Var}(\nu_n^N(1))+ 2 \text{\textnormal{Var}}(\nu_n^N(f)),\nonumber
\end{eqnarray}
where the last step uses the unbiasedness of $\nu_n^N$ proved in Proposition \ref{propVarNu} below. To show (\ref{EMbound2}), observe that by the triangle inequality, by the definition of $\nu_n^N$ and by the Cauchy-Schwarz inequality we have
\begin{eqnarray}
\mathbb{E}[|\eta_n^N(f_n)|] &\leq& \mathbb{E}[|\eta_n^N(f_n)-\nu_n^N(1) \eta_n^N(f_n)|] +\mathbb{E}[|\nu_n^N(f_n)|] \nonumber\\
&\leq& \mathbb{E}[\eta_n^N(f_n)^2]^{\frac12}\mathbb{E}[(\nu_n^N(1)-1)^2]^{\frac12} + \mathbb{E}[\nu_n^N(f_n)^2]^{\frac12}\nonumber.
\end{eqnarray}
Inserting (\ref{EMbound1}) completes the proof.
\end{proof}

Thus we can indeed control the approximation error of $\eta_n^N$ in terms of the approximation error of $\nu_n^N$.\bigskip

The main result of this section shows that $\nu_k^N(f)$ is an unbiased estimator for $\mu_k(f)$ and gives an explicit expression for its variance which is well-suited for deriving our later error bounds. These results are known in the literature, see, e.g., Del Moral and Miclo (\cite{DM00}, Section 2). They are stated here (and proved below) for the sake of completeness and to present exactly what we need in a stringent way and in the form in which we need it. 

\begin{prop}\label{propVarNu}
For all $f\in B(E_n)$, 
\[
\mathbb{E}[\nu_n^N(f)]=\mu_n(f)
\]
and 
\[
\mathbb{E}[|\nu_n^N(f)-\mu_n(f)|^2]=\frac1N \text{\textnormal{Var}}_{\mu_n}(f)+\frac1N \mathbb{E}\left[\sum_{j=0}^{n-1} V_{j,n}^N(f)\right],
\]
where 
\begin{equation}\label{VjnN}
V_{j,n}^N(f)=\nu_j^N(1)\nu_j^N(q_{j,n}(f)^2)-\nu_j^N(q_{j,n}(f))^2+\nu_j^N(q_{j,j+1}(1)-1)\nu_j^N(q_{j,n}(f^2)).
\end{equation}
\end{prop}

The proof of the proposition is based on martingale methods and proceeds in a number of lemmas which make up the remainder of this section. The actual proof of the proposition follows at the end. Note first that for any fixed $f\in B(E_n)$ the process $(A_j)_{j=0}^n$ defined by
\[
A_j=\nu_j^N(q_{j,n}(f))
\]
is an $(\mathcal{F}_j)$-martingale by (\ref{erwetak}) and by the semigroup property of the mappings $q_{j,n}$. Recall that by the Doob decomposition the process $H_j$ given by 
\begin{equation}\label{Doob}
H_j=A_j^2-A_0^2-\sum_{k=0}^{j-1}\mathbb{E}[A_{k+1}^2-A_k^2|\mathcal{F}_k]
\end{equation}
is a martingale. We next derive a more explicit expression for $H_j$.

\begin{lem}\label{lmaN}
We have
\[
H_j=A_j^2-A_0^2-\frac1N \sum_{k=0}^{j-1} \nu_k^N (q_{k,k+1}(1))\nu_k^N (q_{k,k+1}(q_{k+1,n}(f)^2))-\nu_k^N(q_{k,n}(f))^2.
\]
\end{lem}

\begin{proof}[Proof]
By the definitions of $A_k$, $\nu_k^N$ and $\eta_k^N$ and since the random variables $\xi_{k+1}^1,\ldots, \xi_{k+1}^N$ are conditionally (on $\mathcal{F}_k$) independent we can write
\begin{eqnarray}
\mathbb{E}[A_{k+1}^2|\mathcal{F}_k]&=& \frac{\varphi_{k+1}^2}{N^2}\mathbb{E}\left[\left(\left. \sum_{j=1}^N q_{k+1,n}(f)(\xi_{k+1}^j) \right)^2\right|\mathcal{F}_k\right]\nonumber\\
&=& \frac{\varphi_{k+1}^2}{N^2} \left[
\sum_{j=1}^N \mathbb{E}[ q_{k+1,n}(f)(\xi_{k+1}^j)^2 |\mathcal{F}_k]-\sum_{j=1}^N \mathbb{E}[ q_{k+1,n}(f)(\xi_{k+1}^j) |\mathcal{F}_k]^2\right.\nonumber\\  &+&\left. \left(\sum_{j=1}^N \mathbb{E}[ q_{k+1,n}(f)(\xi_{k+1}^j) |\mathcal{F}_k]\right)^2  \right] \nonumber\\
&=& \frac{\varphi_{k+1}^2}{N^2} \left(N \mathbb{E}[ q_{k+1,n}(f)(\xi_{k+1}^1)^2 |\mathcal{F}_k]-N \mathbb{E}[ q_{k+1,n}(f)(\xi_{k+1}^1) |\mathcal{F}_k]^2\right.\nonumber\\
&+&N^2 \left. \mathbb{E}[ q_{k+1,n}(f)(\xi_{k+1}^1) |\mathcal{F}_k]^2\right). \nonumber
\end{eqnarray}
Thus by Lemma \ref{bederweta} and the semigroup property of the $q_{j,k}$ we have
\begin{eqnarray}
\mathbb{E}[A_{k+1}^2|\mathcal{F}_k]&=& \frac{\varphi_{k+1}^2}{N^2} \left[
N \frac{\eta_k^N(q_{k,k+1}(q_{k+1,n}(f)^2))}{\eta_k^N(q_{k,k+1}(1))}-N \frac{\eta_k^N(q_{k,n}(f))^2}{\eta_k^N(q_{k,k+1}(1))^2}+N^2 \frac{\eta_k^N(q_{k,n}(f))^2}{\eta_k^N(q_{k,k+1}(1))^2}\right]\nonumber\\
&=& \frac{1}{N}\bigg[\nu_{k}^N(q_{k,k+1}(1)) \nu_k^N(q_{k,k+1}(q_{k+1,n}(f)^2))- \nu_k^N(q_{k,n}(f))^2\bigg]+A_k^2,\nonumber
\end{eqnarray}
where in the last step we used (\ref{phiketa}). Inserting the resulting expression for $\mathbb{E}[A_{k+1}^2-A_k^2|\mathcal{F}_k]$ into (\ref{Doob}) concludes the proof.
\end{proof}

We can use Lemma \ref{lmaN} to derive an explicit expression for $\mathbb{E}[|\nu_n^N(f)-\mu_n(f)|^2]$. In order to make this expression more tractable, concretely, in order to remove the terms $q_{k,k+1}(q_{k+1,n}(f)^2)$, we use the following lemma:\bigskip

\begin{lem}\label{lmaMartLM}
For all $f\in B(E_n)$ the processes
\begin{eqnarray}\label{frmLk}
L_k&=&\nu_k^N(1)\nu_k^N(q_{k,n}(f)^2)-\nu_0^N(1) \nu_0^N(q_{0,n}(f)^2)\nonumber\\&-&\sum_{j=0}^{k-1} \nu_j^N (q_{j,j+1}(1))\nu_j^N (q_{j,j+1}(q_{j+1,n}(f)^2))
+\sum_{j=0}^{k-1} \nu_j^N(1)\nu_j^N(q_{j,n}(f)^2)
\end{eqnarray}
and
\begin{equation}\label{frmMk}
M_k=\nu_k^N(1)\nu_k^N(q_{k,n}(f^2))-\nu_0^N(1)\nu_0^N(q_{0,n}(f^2))-\sum_{j=0}^{k-1}\nu_{j}^N(q_{j,j+1}(1)-1)\nu_j^N(q_{j,n}(f^2))
\end{equation}
are $(\mathcal{F}_k)$-martingales.
\end{lem}
\begin{proof}
By the Doob decomposition, for $B_k=\nu_k^N(1)\nu_k^N(q_{k,n}(f)^2)$ the process
\[
L_k=B_k-B_0-\sum_{j=0}^{k-1}\mathbb{E}[B_{j+1}|\mathcal{F}_j]+\sum_{j=0}^{k-1}B_{j}
\]
is a martingale. To obtain the expression in (\ref{frmLk}) it is sufficient to note that by Lemma \ref{bederweta} and by (\ref{phiketa})
\begin{eqnarray}
\mathbb{E}[B_{j+1}|\mathcal{F}_j] &=& \varphi_{j+1}^2 \mathbb{E}[\eta_{j+1}^N(q_{j+1,n}(f)^2)|\mathcal{F}_j]= \varphi_{j+1}^2 \frac{\eta_{j}^N (q_{j,j+1}(q_{j+1,n}(f)^2))}{\eta_{j}^N (q_{j,j+1}(1))}\nonumber\\
&=& \nu_j^N(q_{j,j+1}(1)) \nu_j^N(q_{j,j+1}(q_{j+1,n}(f)^2)).\nonumber
\end{eqnarray}
Likewise for $\widetilde{B}_k= \nu_k^N(1)\nu_k^N(q_{k,n}(f^2))$ the process
\[
M_k=\widetilde{B}_k-\widetilde{B}_0-\sum_{j=0}^{k-1}\mathbb{E}[\widetilde{B}_{j+1}-\widetilde{B}_j|\mathcal{F}_j]
\]
is a martingale. To obtain the expression in (\ref{frmMk}) note that by (\ref{erwetak}) and by (\ref{phiketa})
\begin{eqnarray}
\mathbb{E}[\widetilde{B}_{j+1}-\widetilde{B}_j|\mathcal{F}_j]&=& \nu_{j+1}^N(1) \mathbb{E}[\nu_{j+1}^N(q_{j+1,n}(f^2)) |\mathcal{F}_j]-\nu_j^N(1)\nu_j^N(q_{j,n}(f^2))\nonumber\\
&=& (\nu_{j+1}^N(1)-\nu_{j}^N(1))\nu_j^N(q_{j,n}(f^2))\nonumber\\
&=&\nu_{j}^N(q_{j,j+1}(1)-1)\nu_j^N(q_{j,n}(f^2)).\nonumber
\end{eqnarray}
\end{proof}

With these lemmas, we have established the tools needed to prove Proposition \ref{propVarNu}: 

\begin{proof}[Proof of Proposition \ref{propVarNu}]
For the unbiasedness, note that 
\begin{eqnarray}
\nu_n^N(f)-\mu_n(f)&=&\nu_n^N(q_{n,n}(f))-\nu_0^N(q_{0,n}(f))+\nu_0^N(q_{0,n}(f))-\mu_0(q_{0,n}(f))\nonumber\\
&=&A_n-A_0+\nu_0^N(q_{0,n}(f))-\mu_0(q_{0,n}(f)).\nonumber
\end{eqnarray}
This implies
\[
\mathbb{E}[\nu_n^N(f)]=\mathbb{E}[\mu_n(f)]
\]
since the martingale property of $A_n$ implies $\mathbb{E}[A_n-A_0]=0$, and since we have
\[
\mathbb{E}[\nu_0^N(q_{0,n}(f))]=\mathbb{E}[\mu_0(q_{0,n}(f))]
\]
because of $\nu_0^N=\eta_0^N$ and because the particles $\xi_0^j$ are independent samples from $\mu_0$.\bigskip

Now note that by conditional independence
\begin{eqnarray}
\mathbb{E}[|\nu_n^N(f)-\mu_n(f)|^2]&=&\mathbb{E}[|(A_n-A_0)+ (\nu_0^N(q_{0,n}(f))-\mu_0(q_{0,n}(f)))|^2]\nonumber\\
&=& \mathbb{E}[A_n^2-A_0^2]+\mathbb{E}[(\nu_0^N(q_{0,n}(f))-\mu_0(q_{0,n}(f)))^2].
\end{eqnarray}
Note that the second summand equals $\frac{1}{N}\text{Var}_{\mu_0}(q_{0,n}(f))$. Using Lemma \ref{lmaN} and, in the second step, (\ref{frmLk}) we can thus write
\begin{eqnarray}\label{fastfertig}
\mathbb{E}[|\nu_n^N(f)-\mu_n(f)|^2]&=&\frac{1}{N}\Big( \text{Var}_{\mu_0}(q_{0,n}(f))\nonumber\\
&+&E\Big[\sum_{k=0}^{n-1} \nu_k^N (q_{k,k+1}(1))\nu_k^N (q_{k,k+1}(q_{k+1,n}(f)^2))-\nu_k^N(q_{k,n}(f))^2\Big]\Big)\nonumber\\
&=&\frac{1}{N}\Big( \text{Var}_{\mu_0}(q_{0,n}(f)) + E\Big[\nu_n^N(1)\nu_n^N(q_{n,n}(f)^2)-\nu_0^N(1) \nu_0^N(q_{0,n}(f)^2) \nonumber\\ 
&+&\sum_{k=0}^{n-1} \nu_k^N(1)\nu_k^N(q_{k,n}(f)^2)-\nu_k^N(q_{k,n}(f))^2\Big] \Big).
\end{eqnarray}
Now observe that
\begin{eqnarray}\label{varmu1}
\text{Var}_{\mu_0}(q_{0,n}(f)) - \mathbb{E}[\nu_0^N(1) \nu_0^N(q_{0,n}(f)^2)]&=&\text{Var}_{\mu_0}(q_{0,n}(f)) - \mu_0(q_{0,n}(f)^2)\nonumber\\
&=&-\mu_0(q_{0,n}(f))^2=-\mu_n(f)^2.
\end{eqnarray}
Moreover, by (\ref{frmMk}) we have
\begin{eqnarray}\label{varmu2}
\mathbb{E}[\nu_n^N(1)\nu_n^N(f^2)]-\mu_n(f)^2&=& \mathbb{E}[\nu_0^N(1)\nu_0^N(q_{0,n}(f^2))]-\mu_n(f)^2\nonumber\\
&+&\sum_{j=0}^{n-1}\mathbb{E}[\nu_j^N(q_{j,j+1}(1)-1)\nu_j^N(q_{j,n}(f^2))]\nonumber\\
&=&\text{Var}_{\mu_n}(f) + \sum_{j=0}^{n-1}\mathbb{E}[\nu_j^N(q_{j,j+1}(1)-1)\nu_j^N(q_{j,n}(f^2))].\nonumber\\
\end{eqnarray}
Inserting (\ref{varmu1}) and then (\ref{varmu2}) into (\ref{fastfertig}) yields
\[
\mathbb{E}[|\nu_n^N(f)-\mu_n(f)|^2]=\frac{1}{N}\left( \text{Var}_{\mu_n}(f)+\mathbb{E}\left[\sum_{j=0}^{n-1}V_{j,n}^N(f)\right]\right)
\]
with
\[
V_{j,n}^N(f)=\nu_j^N(1)\nu_j^N(q_{j,n}(f)^2)-\nu_j^N(q_{j,n}(f))^2
+\nu_j^N(q_{j,j+1}(1)-1)\nu_j^N(q_{j,n}(f^2))
\]
so we are done.
\end{proof}

Throughout, we assume the integrands $f$ to be bounded. This assumption becomes crucial when transferring results from the weighted particle measure to the original one in Lemma \ref{EMbounds}. For this reason, we restrict attention to bounded integrands, avoiding to introduce systems of integrability conditions (as is done, e.g., in Chapter 9.2 of Capp\'e, Moulines and Ryd\'en, \cite{CMR05}). However, the results on the weighted particle measures $\nu_n^N$ do not rely on this boundedness and could be generalized along such lines.

\section{Non-asymptotic Error Bounds}\label{secQuantCBG}
In this section we derive non-asymptotic error bounds from our expression for the variance of $\nu_n^N(f)$ derived in the previous section. For $0 \leq j \leq n$, let $\|\cdot\|_j$ be a norm on the function space $B(E_j)$ such that $\|f\|_j<\infty$ for all $f\in B(E_j)$. We obtain our bounds from the following assumption which arises naturally in the proofs. In the later sections, we show how to verify the assumption from more elementary conditions in more specific settings.

\begin{ass}\label{cjkbound} 
For all $0 \leq j <k \leq n$, there exists a constant $c_{j,k}$ such that for all $f\in B(E_k)$, the following stability inequality for the propagator $q_{j,k}$ is satisfied:
\begin{equation}\label{cjninequality}
\max(\|1\|_j\|q_{j,k}(f)^2\|_j,\|q_{j,k}(f)\|_j^2,\|q_{j,k}(f^2)\|_j)  \leq c_{j,k}\|f\|_k^2.
\end{equation}
\end{ass}
In the following we show that the quadratic approximation error of $\nu_n^N$ can essentially be controlled through the constants $c_{j,k}$.\bigskip 

Recall that the approximation error was given by the expected sum over $j$ of the expressions $V_{j,n}^N$ defined in the previous sections. We first show how $V_{j,n}^N$ can be bounded through  $V_{j,n}$ defined by
\[
V_{j,n}(f)=\text{Var}_{\mu_j}(q_{j,n}(f))
\]
and an error term. Note that $V_{j,n}$ is what we obtain when substituting $\nu_{j}^N$ by $\mu_j$ in our expression for $V_{j,n}^N$. Define
\[
\varepsilon_j^{N}=\sup\Big\{
\mathbb{E}[|\nu_j^N(f)-\mu_j(f)|^2]\, \Big|\, \|f\|_j\leq 1
 \Big\}.
\]
Then we have the following result:\bigskip

\begin{prop}\label{propVjncjn}
For $0\leq j <n $ we have
\[
\mathbb{E}[V_{j,n}^N(f)]\leq V_{j,n}(f) + c_{j,n} \| f\|_n^2\Big(2+ \| q_{j,j+1}(1)-1 \|_j\Big) \varepsilon_j^{N}.
\]
\end{prop}

\begin{proof}
Note first that by the Cauchy-Schwarz inequality and since $\nu_j^N(\cdot)$ is an unbiased estimator for $\mu_j(\cdot)$, we have for any $g,h\in B(E_j)$
\begin{eqnarray}\label{abschfgc}
|\mathbb{E}[\nu_j^N(g)\nu_j^N(h)-\mu_j(g)\mu_j(h)]|\nonumber \\&\leq&
|\mu_j(g) \mathbb{E}[\nu_j^N(h)-\mu_j(h)]+\mu_j(h)\mathbb{E}[\nu_j^N(g)-\mu_j(g)]|\nonumber \\
&+&\mathbb{E}[|\nu_j^N(g)-\mu_j(g)||\nu_j^N(h)-\mu_j(h)|]
\nonumber \\&\leq&   \| g\|_j \| h\|_j\;\varepsilon_j^{N}.
\end{eqnarray}
Adding $\pm V_{j,n}(f)$ to the definition (\ref{VjnN}) of $V_{j,n}^N(f)$ and applying (\ref{abschfgc}) three times yields
\[
\mathbb{E}[V_{j,n}^N(f)] \leq V_{j,n}(f) + R_{j,n}(f) \;\varepsilon_j^{N} 
\]
with
\begin{eqnarray}
R_{j,n}(f)&=&\|1\|_j\|q_{j,n}(f)^2\|_j+\|q_{j,n}(f)\|_j^2+\|q_{j,n}(f^2)\|_j
\|q_{j,j+1}(1)-1\|_j.\nonumber
\end{eqnarray}
Applying (\ref{cjninequality}) yields
\[
R_{j,n}(f)\leq c_{j,n}\|f\|_n^2(2+\|q_{j,j+1}(1)-1\|_j)
\]
and thus the desired inequality.
\end{proof}

In order to state the main result of this section, we need a few more definitions. Define $\widehat{c}_{k}$ and $\widehat{v}_k$ by
\[
\widehat{c}_{k}=\sum_{j=0}^{k-1}c_{j,k} \Big(2+\| q_{j,j+1}(1)-1 \|_j\Big)
\]
and
\[
\widehat{v}_{k}=\sup\left\{\left.
\sum_{j=0}^k \text{Var}_{\mu_j}(q_{j,k}(f)) \right| \|f\|_k\leq 1
 \right\}.
\]
Furthermore define
\[
\overline{c}_k=\max_{j\leq k} \widehat{c}_{j},\;\;\;\;\; \overline{v}_k=\max_{j\leq k} \widehat{v}_{j}\;\;\; \text{ and }\;\;\; \overline{\varepsilon}_k^{N}= \max_{j\leq k} \varepsilon_j^N.
\]
Then we have the following bound on the approximation error:\bigskip

\begin{thm}\label{thmBound}
For $f\in B(E_n)$ we have
\begin{equation}\label{thmf1c}
N \mathbb{E}[|\nu_n^N(f)-\mu_n(f)|^2] \leq \sum_{j=0}^n \text{\textnormal{Var}}_{\mu_j}(q_{j,n}(f))+\| f \|_n^2 \widehat{c}_{n}\;  \overline{\varepsilon}_n^{N}.
\end{equation}
Moreover, if $N\geq 2 \overline{c}_n$ we have
\begin{equation}\label{thmf2c}
\overline{\varepsilon}_n^{N}\leq 2\frac{\overline{v}_n}{N}.
\end{equation}
\end{thm}

\begin{proof}[Proof of Theorem \ref{thmBound}]
Note that by Propositions \ref{propVarNu} and \ref{propVjncjn} and by the definition of $V_{j,n}(f)$ we get
\begin{eqnarray}
&&N \mathbb{E}[|\nu_n^N(f)-\mu_n(f)|^2]\nonumber\\ &\leq& \sum_{j=0}^n\text{Var}_{\mu_j}(q_{j,n}(f)) +\| f \|_{n}^2 
\sum_{j=0}^{n-1} c_{j,n} (2+\|q_{j,j+1}(1)-1 \|_j )\;\varepsilon_j^{N}\nonumber.
\end{eqnarray}
Bounding $\varepsilon_j^{N}$ by $\overline{\varepsilon}_n^{N}$ and inserting the definition of $\widehat{c}_n$ shows (\ref{thmf1c}). Optimizing (\ref{thmf1c}) over $f$ with $\|f\|_n\leq 1$ and over $n$ yields
\[
N \overline{\varepsilon}_n^{N} \leq \overline{v}_n +\overline{c}_n \; \overline{\varepsilon}_n^{N}. 
\]
Choosing $N\geq 2\,\overline{c}_n$ and thus $N-\overline{c}_n \geq \frac{N}{2}$ gives (\ref{thmf2c}).
\end{proof}

\section{$L_p$-stability}\label{LPglobal}

In this section, we show how $L_p$-stability of the Feynman-Kac propagators $q_{j,k}$ can be derived from global mixing properties of the MCMC kernels $K_k$. We now consider a more restricted setting where the measures $\mu_k$ live on a common state space and where the transition kernels $K_k$ are stationary with respect to the measures $\mu_k$.\bigskip

\subsection{The setting}\label{Lpsetting}

We now add the assumption that all the distributions $\mu_k$ live on the same support $E_k=E$. Moreover, we assume stationarity of the kernels $K_k$: For $1 \leq k \leq n$, let $K_k(x,A)$ be an integral operator with $K_k(\cdot,f)\in B(E)$ for all $f\in B(E)$, with $K_k(x,\cdot)\in M_1(E)$ for all $x\in E$ and with stationary distribution $\mu_k$, i.e.,
\[
\mu_k K_k (A)=\mu_k (A)\;\;\;\;\;\;\text{for all}\;\;\;\;A\in \mathcal{B}(E).
\]
Thus, $K_k$ can be thought of, e.g., as many steps of a Metropolis chain with respect target $\mu_k$. Accordingly, the $\mu_k$ are related through
\[
\mu_k(f)= \frac{\mu_{k-1}(g_{k-1,k} f)}{\mu_{k-1}(g_{k-1,k})}
\]
so that the functions $g_{k-1,k}\in B(E)$ are (unnormalized) relative densities between $\mu_{k-1}$ and $\mu_k$. In the notation of Section \ref{tmvmodel} this corresponds to $\hat{\mu}_k=\mu_k$ for all $k$.\bigskip

We now want to apply the error bound of Theorem \ref{thmBound}. We thus need to define a series of norms $\|\cdot\|_j$ on $B(E)$ and find constants $c_{j,k}$ such that the inequality 
\begin{equation}\label{cjkinequalityp}
\max\left(\|q_{j,k}(f)\|_j^2,\|q_{j,k}(f)^2\|_j,\|q_{j,k}(f^2)\|_j\right)  \leq c_{j,k}\;\|f\|_k^2
\end{equation}
from Assumption \ref{cjkbound} is satisfied for all $f\in B(E)$.\bigskip

We fix $p>2$ and choose $\|\cdot\|_j=\|\cdot\|_{L_p(\mu_j)}$, i.e.,
\[
\|f\|_j=\|f\|_{L_p(\mu_j)}=\mu_j(|f|^p)^{\frac1p}.
\]
Note that all bounded measurable functions are in $L_p(\mu_j)$ and that for $f \in B(E)$ we have $q_{j,k}(f)\in B(E)$ since we assumed the relative densities $g_{k-1,k}$ to be bounded. Now define $\widetilde{c}_{j,k}(p,q)$ to be the constant in an $L_p$-$L_q$-bound for $q_{j,k}$: For $p>  q> 1$ and $0\leq j< k\leq n$ we have
\[
\| q_{j,k}(f)\|_{L_p(\mu_j)} \leq
\widetilde{c}_{j,k}(p,q) \| f\|_{L_{q}(\mu_k) } \;\;\text{ for all } f\in B(E).
\]
Such constants will be studied in Section \ref{secStabFK} below. The next lemma shows how the constants $\widetilde{c}_{j,k}$ can be used to verify a version of Assumption \ref{cjkbound}:
\begin{lem}\label{cjklemma}
Define for some $p>2$ and all $0\leq j< k\leq n$
\[
c_{j,k}(p)=\max\left(\widetilde{c}_{j,k}\left(p,\frac{p}{2}\right),\;\;\widetilde{c}_{j,k}(2p,p)^2\right).
\]
Then (\ref{cjkinequalityp}) is satisfied with $\|\cdot\|_j=\|\cdot\|_{L_p(\mu_j)}$ and with constants $c_{j,k}(p)$, i.e., for all $f\in B(E)$ we have
\[
\max\left(\|1\|_{L_p(\mu_j)} \|q_{j,k}(f)^2\|_{L_p(\mu_j)},\|q_{j,k}(f)\|_{L_p(\mu_j)}^2,\|q_{j,k}(f^2)\|_{L_p(\mu_j)} \right)  \leq c_{j,k}(p)\;\|f\|_{L_p(\mu_k)}^2.
\]
\end{lem}
\begin{proof}
Observe first that we have
\[
\|q_{j,k}(f^2)\|_{L_p(\mu_j)} \leq \widetilde{c}_{j,k}\left(p,\frac{p}2\right) \, \|f^2\|_{L_\frac{p}{2}(\mu_k)}=\widetilde{c}_{j,k}\left(p,\frac{p}2\right) \, \|f\|_{L_p(\mu_k)}^2.
\]
Furthermore we have $\|1\|_{L_p(\mu_j)}=1$ and
\[
\|q_{j,k}(f)^2\|_{L_p(\mu_j)}=\|q_{j,k}(f)\|_{L_{2p}(\mu_j)}^2 \leq \widetilde{c}_{j,k}(2p,p)^2 \|f\|_{L_{p}(\mu_k)}^2.
\]
Finally, observing that 
\[
\|q_{j,k}(f)\|_{L_{p}(\mu_j)}^2\leq \|q_{j,k}(f)\|_{L_{2p}(\mu_j)}^2
\]
concludes the proof.
\end{proof}

\subsection{$L_p$-Stability of Feynman-Kac Propagators}\label{secStabFK}

In this section we show how to derive inequalities of the type
\begin{equation}\label{zielineq}
\| q_{j,k}(f)\|_{L_p(\mu_j)} \leq \widetilde{c}_{j,k}(p,q)\|f\|_{L_q(\mu_k)}\;\;\;\;\text{for}\;\;\;\; p \leq q, j<k 
\end{equation}
from suitable mixing conditions on the kernels $K_j,\ldots, K_k$. These are exactly the inequalities we need in order to make the error bounds of the previous section explicit. The constants $\widetilde{c}_{j,k}$ we derive are independent of the length of the time interval $k-j$. The central intermediate step is Theorem \ref{propLpIt} which derives (\ref{zielineq}) for the case $p=q=2^r$, $r\in \mathbb{N}$, and for a modified propagator $\hat{q}_{j,k}$ from an $L_2$-mixing condition for the kernels $K_j,\ldots, K_k$. The proof of Theorem \ref{propLpIt} proceeds in a number of lemmas starting with only one step, $k-j=1$, and $p=2$ and then gradually generalizing the result by showing how to proceed from $L_p$-stability to $L_{2p}$-stability and to more than one step, $k-j >1$. Theorem \ref{propLpIt} is followed by a number of corollaries, showing how to transfer the result to the original propagator $q_{j,k}$, and, using an additional assumption of hyperboundedness, to the case $p>q$. Corollary \ref{corNEULP} collects these observations and states the resulting version of inequality (\ref{zielineq}). We close the section with an additional result, Proposition \ref{expdecay}, which shows that for functions $f$ with $\mu_k(f)=0$ we can obtain a version of inequality (\ref{zielineq}) where the constants $\widetilde{c}_{j,k}(p,q)$ decay exponentially in the length of the time interval $k-j$.\bigskip

Throughout this section we assume a uniform upper bound $\gamma$ on the normalized relative densities $\overline{g}_{k-1,k}$ given by 
\[
\overline{g}_{k-1,k}=\frac{g_{k-1,k}}{\mu_{k-1}(g_{k-1,k})}.
\]

\begin{ass}\label{ass1}
There exists $\gamma>1$ such that
\begin{equation}\label{gammabdbd}
\overline{g}_{k-1,k}(x)\leq \gamma
\end{equation}
for all $x\in E$ and all $k$ with $1\leq k \leq n$. 
\end{ass}

$\gamma$ is a rough measure of how strongly the measures $\mu_k$ differ from each other. For the analysis of this section it proves to be convenient not to work with $q_{j,k}$ directly but to work with $\hat{q}_{j,k}$ defined as follows: For $1 \leq k \leq n$ define  $\hat{q}_{k-1,k}:B(E)\rightarrow B(E)$ by
\[
\hat{q}_{k-1,k}(f)=K_{k-1}\left(\overline{g}_{k-1,k} f \right). 
\]
Furthermore, define for $1 \leq j < k \leq n$ the mapping $\hat{q}_{j,k}:B(E)\rightarrow B(E)$ by
\[
\hat{q}_{j,k}(f)=\hat{q}_{j,j+1}(\hat{q}_{j+1,j+2}(\ldots \hat{q}_{k-1,k}(f)))\;\;\;\text{ and }\;\; \hat{q}_{k,k}(f)=f.
\]
By definition $\hat{q}_{j,k}$ is a semigroup. $q_{j,k}$ and $\hat{q}_{j+1,k}$ are related through 
\[
q_{j,k}(f)=\overline{g}_{j,j+1}\hat{q}_{j+1,k}(K_{k}(f)).
\]
Stability results for $\hat{q}_{j,k}$ can be transfered to $q_{j,k}$ using Lemma \ref{lemqhatq} which is stated and proved later in this section.\bigskip

$L_1$-stability of $\hat{q}_{j,k}$ simply follows from
\begin{equation}\label{L1it}
\| \hat{q}_{j,k}(f) \|_{L_1(\mu_j)}=\mu_j(|\hat{q}_{j,k}(f)|)\leq \mu_j(\hat{q}_{j,k}(|f|))=\mu_k(|f|)=\| f \|_{L_1(\mu_k)}.
\end{equation}

From (\ref{gammabdbd}) and from the fact that $\mu_k$ is a stationary distribution for $K_k$ it is easy to conclude bounds such as
\[
\| \hat{q}_{j,k}(f) \|_{L_2(\mu_j)}^2\leq \gamma^{k-j}\| f \|_{L_2(\mu_k)}^2.
\]
This bound has the major disadvantage that it degenerates exponentially in $k-j$ since $\gamma>1$. In the following, we assume and  exploit mixing properties of the kernels $K_k$ in order to obtain $L_p$-bounds for $p>1$ which do not degenerate in $k-j$.\bigskip 

Throughout this section, we assume the following mixing condition: 
\begin{ass}\label{ass2}
There are constants $\alpha\in (0,1)$ and $\beta \in [0,1]$ such that for all $f\in B(E)$ we have the following $L_2$-bound for $\hat{q}_{k-1,k}$:
\begin{equation}\label{alphabetabd}
\| \hat{q}_{k-1,k}(f) \|_{L_2(\mu_{k-1})}^2\leq \alpha \| f \|_{L_2(\mu_{k})}^2 + \beta \mu_k(f)^2.
\end{equation}
\end{ass}
Additionally, our results will impose conditions that $\alpha$ is sufficiently small. In Section \ref{secStabFKP} below it is shown that one way to ensure that (\ref{alphabetabd}) holds with a sufficiently small $\alpha$ is to assume that the kernels $K_k$ satisfy Poincaré inequalities associated with a sufficiently large spectral gap.\bigskip 

Our first step is to iterate (\ref{alphabetabd}) in order to obtain $L_2$-bounds for $\hat{q}_{j,k}$: 

\begin{lem}\label{lemL2It}
For $1\leq j< k \leq n$ and $f\in B(E)$ we have the bounds
\begin{equation}\label{L2It1}
\| \hat{q}_{j,k}(f) \|_{L_2(\mu_{j})}^2\leq 
\alpha^{k-j}\, \| f \|_{L_2(\mu_{k})}^2 +\frac{\beta}{1-\alpha }\, \mu_k(f)^2 
\end{equation}
and
\begin{equation}\label{L2It2}
\| \hat{q}_{j,k}(f) \|_{L_2(\mu_{j})}\leq 
 \frac{1}{(1-\alpha )^{\frac12}}\, \| f \|_{L_2(\mu_{k})}.
\end{equation}
\end{lem}

\begin{proof}
Iterating the bound (\ref{alphabetabd}) and utilizing that $\mu_j(\hat{q}_{j,k}(f))=\mu_k(f)$ we get
\[
\| \hat{q}_{j,k}(f) \|_{L_2(\mu_{j})}^2 \leq \alpha^{k-j}\| f \|_{L_2(\mu_{k})}^2+\sum_{i=0}^{k-j-1}\beta \alpha^i \mu_k(f)^2.
\]
Applying to this the geometric series inequality immediately implies (\ref{L2It1}). Furthermore, since we assumed $\beta \leq 1 $ and $\mu_k(f)^2\leq \| f \|_{L_2(\mu_{k})}^2$ we have
\[
\| \hat{q}_{j,k}(f) \|_{L_2(\mu_{j})}^2 \leq \sum_{i=0}^{k-j}\alpha^i \| f \|_{L_2(\mu_{k})}^2.
\]
Applying again the geometric series inequality yields (\ref{L2It2}).
\end{proof}

We now turn to $L_p$-bounds for the case of $p=2^r$ with $r\in \mathbb{N}$. We proceed inductively, deducing the bound for $p=2^r$ from the bound for $p=2^{r-1}$. We begin by deriving an $L_{2p}$-bound for $\hat{q}_{k-1,k}$ from (\ref{alphabetabd}).\bigskip

\begin{lem}\label{lemL2pS1}
For $1\leq k \leq n$, $f\in B(E)$ and $p\geq 1$ we have
\[
\| \hat{q}_{k-1,k}(f) \|_{L_{2p} (\mu_{k-1})}^{2p}\leq \alpha \,\gamma^{2p-2} \| f \|_{L_{2p}(\mu_{k})}^{2p} +\beta \,\gamma^{2p-2} \| f \|_{L_p(\mu_{k})}^{2p}.
\]
\end{lem}
\begin{proof}
Note that we have
\begin{eqnarray}
 \| \hat{q}_{k-1,k}(f) \|_{L_{2p} (\mu_{k-1})}^{2p} &=& \mu_{k-1}(|K_{k-1}(\overline{g}_{k-1,k}f)|^{2p})\nonumber\\
&\leq& \mu_{k-1}(K_{k-1}(\overline{g}_{k-1,k}^p|f|^p)^2)\nonumber\\&=& \| \hat{q}_{k-1,k}(\overline{g}_{k-1,k}^{p-1} |f|^p) \|_{L_2 (\mu_{k-1})}^2. \nonumber
\end{eqnarray}
Applying now the bound (\ref{alphabetabd}) yields
\begin{eqnarray}
\| \hat{q}_{k-1,k}(f) \|_{L_{2p} (\mu_{k-1})}^{2p} &\leq& \alpha \| \overline{g}_{k-1,k}^{p-1} |f|^p \|_{L_2 (\mu_{k})}^{2}+\beta \| \overline{g}_{k-1,k}^{p-1} |f|^p \|_{L_1 (\mu_{k})}^{2}  \nonumber\\
&\leq& \alpha \gamma^{2p-2} \| f \|_{L_{2p} (\mu_{k})}^{2p} +\beta \gamma^{2p-2} \| f \|_{L_p (\mu_{k})}^{2p}.
\end{eqnarray}
\end{proof}

Our next step is to show how by applying Lemma \ref{lemL2pS1} we get an $L_{2p}$-bound for $\hat{q}_{j,k}$ from an $L_{p}$-bound.\bigskip

\begin{lem}\label{lemL2pLpIt}
Assume that $\alpha \gamma^{2p-2}<1$ and that for $\delta(p)\geq 1$ we have for $1\leq j< k \leq n$ and $f\in B(E)$ the inequality
\[
\|\hat{q}_{j,k}(f)\|_{L_{p} (\mu_{j})}\leq \delta(p)\|f\|_{L_{p} (\mu_{k})}. 
\]
Then we have
\[
\|\hat{q}_{j,k}(f)\|_{L_{2p} (\mu_{j})}\leq \delta(2p)\|f\|_{L_{2p} (\mu_{k})} 
\]
with
\[
\delta(2p)=\delta(p) \frac{\gamma^{1-\frac1p}}{(1-\alpha \gamma^{2p-2})^{\frac{1}{2p}}}.
\]
\end{lem}

\begin{proof}
Define $\theta=\alpha \gamma^{2p-2}$. Iterating the inequality of Lemma \ref{lemL2pS1} and utilizing that $\beta \leq 1$, we get
\begin{equation}\label{L2pIneqIt}
\|\hat{q}_{j,k}(f)\|_{L_{2p} (\mu_{j})}^{2p}\leq \theta^{k-j} \|f\|_{L_{2p} (\mu_{k})}^{2p}+\gamma^{2p-2}\sum_{i=j+1}^{k}\theta^{i-1-j} \|\hat{q}_{i,k}(f)\|_{L_{p} (\mu_{i})}^{2p}
\end{equation}
Using our assumption on $\|\hat{q}_{j,k}\|_{L_{p}(\mu_{j})}$ and the facts that $\|f\|_{L_{p}(\mu_{k})}\leq \|f\|_{L_{2p}(\mu_{k})}$, $\gamma\geq 1$ and $\delta(p)\geq 1$, we get that
\[
\|\hat{q}_{j,k}(f)\|_{L_{2p} (\mu_{j})}^{2p}\leq \|f\|_{L_{2p} (\mu_{k})}^{2p} \gamma^{2p-2}\delta(p)^{2p}  \sum_{i=0}^{k-j}\theta^{i}. 
\]
Thus, since we assumed $\theta<1$, by the geometric series inequality we have
\[
\|\hat{q}_{j,k}(f)\|_{L_{2p} (\mu_{j})}\leq \delta(2p)\|f\|_{L_{2p} (\mu_{k})} 
\]
with
\[
\delta(2p)=\delta(p) \frac{\gamma^{1-\frac1p}}{(1-\alpha\gamma^{2p-2})^{\frac{1}{2p}} }.
\]
\end{proof}

Combining Lemmas \ref{lemL2It} and \ref{lemL2pLpIt} we can state the key result of this section as follows:\bigskip

\begin{thm}\label{propLpIt}
For $r\in \mathbb{N}$, consider $p=2^r$ and assume that $\alpha \gamma^{p-2}<1$. Then we have for $1\leq j< k \leq n$ and $f\in B(E)$ the inequality
\[
\|\hat{q}_{j,k}(f)\|_{L_{p} (\mu_{j})}\leq \delta(p)\|f\|_{L_{p} (\mu_{k})} 
\]
with
\[
\delta(p)=\prod_{i=1}^r \frac{\gamma^{1-2^{-(i-1)}}}{(1-\alpha \gamma^{2^i-2})^{2^{-i}}}< \frac{\gamma^{r-2+2^{-(r-1)}}}{1-\alpha\gamma^{2^r-2}}
\]
\end{thm}

\begin{proof}
The case $r=0$ follows from (\ref{L1it}). In the case $r=1$, the inequality coincides with (\ref{L2It2}). The inequalities for $r > 1$ follow because Lemma \ref{lemL2pLpIt} implies that we can choose
\[
\delta(2^r)=\delta(2) \prod_{i=2}^r \frac{\gamma^{1-2^{-(i-1)}}}{(1-\alpha \gamma^{2^i-2})^{2^{-i}}}.
\]
We can apply Lemma \ref{lemL2pLpIt} iteratively, since $\alpha \gamma^{p-2}<1$ implies $\alpha \gamma^{q-2}<1$ for all $q\leq p$.  For the upper bound on $\delta(p)$, we apply the geometric series equality in the nominator, bound the term in brackets under the exponent in the denominator by $1-\alpha \gamma^{p-2}$ and apply the geometric series inequality to the product.
\end{proof}

Since the constants $\delta(2^r)$ are monotonically increasing in $r$, we can immediately extend the bounds of Theorem  \ref{propLpIt} to general $p\geq 1$ using the Riesz-Thorin interpolation theorem (see Davies \cite{DA90}, §1.1.5):\bigskip

\begin{cor}\label{corLp}
Consider $p\in [2^r,2^{r+1}]$ for $r\in \mathbb{N}$ and assume $\alpha \gamma^{2^{r+1}-2}<1$. Then for $1\leq j< k \leq n$ and $f\in B(E)$ we have
\[
\|\hat{q}_{j,k}(f)\|_{L_{p} (\mu_{j})}\leq \delta(p)\|f\|_{L_{p} (\mu_{k})} 
\]
with $\delta(p)$ given by
\[
\delta(p)=\delta(2^{r+1}),
\]
where $\delta(2^{r+1})$ is defined as in Theorem $\ref{propLpIt}$.
\end{cor}

Since we need $L_{2p}$-$L_p$-bounds in the error bounds of Section \ref{secQuantCBG} we now show that given that we have an $L_p$-$L_q$-bound for $K_k$, we can immediately conclude from Corollary \ref{corLp} an $L_p$-$L_q$-bound for $\hat{q}_{j,k}$:\bigskip

\begin{cor}\label{corHypqhat}
Consider $p\geq 1$ and $q\geq 1$. Let $q\in [2^r,2^{r+1}]$ for $r\in \mathbb{N}$ and assume $\alpha \gamma^{2^{r+1}-2}<1$. Assume that for $1\leq j< n$ we have a constant $\theta_j(p,q)\geq 0$ such that
\begin{equation}\label{hyperc}
\|K_{j}(f)\|_{L_{p} (\mu_{j})}\leq \theta_j(p,q)\|f\|_{L_{q} (\mu_{j})}. 
\end{equation}
Then for $j< k \leq n$ and $f\in B(E)$ we have
\[
\|\hat{q}_{j,k}(f)\|_{L_{p} (\mu_{j})}\leq \theta_j(p,q)\gamma^{\frac{q-1}{q}}\delta(q)\|f\|_{L_{q} (\mu_{k})} 
\]
with $\delta(q)$ as defined in Corollary \ref{corLp}.
\end{cor}
\begin{proof}
By (\ref{hyperc}) we have
\[
\|\hat{q}_{j,k}(f)\|_{L_{p} (\mu_{j})}\leq \theta_j(p,q)\|\overline{g}_{j,j+1} \hat{q}_{j+1,k}(f)\|_{L_{q} (\mu_{j})} 
\]
and thus by Corollary $\ref{corLp}$
\[
\|\hat{q}_{j,k}(f)\|_{L_{p} (\mu_{j})}\leq \theta_j(p,q)\gamma^{\frac{q-1}{q}}\delta(q)\|f\|_{L_{q} (\mu_{k})}. 
\]
\end{proof}
\bigskip

The following lemma shows, that $L_p$-$L_q$-bounds for the $\hat{q}_{j,k}$ can be used to obtain $L_p$-$L_q$-bounds for the original mappings $q_{j,k}$.\bigskip

\begin{lem}\label{lemqhatq}
Assume that for some $p\geq 1$ and $q\geq 1$ we have a $\delta\geq 0$ such that for all $f\in B(E)$ and for all $1\leq j< k\leq n$
\[
\| \hat{q}_{j,k}(f) \|_{L_p(\mu_j)} \leq \delta\, \| f \|_{L_q(\mu_k)}.
\]
Then we have
\[
\| q_{j,k}(f) \|_{L_p(\mu_j)} \leq \delta\, \gamma^{\frac{p-1}{p}} \| f \|_{L_q(\mu_k)}.
\]
\end{lem}

\begin{proof}
Note that we have
\begin{eqnarray}
\| q_{j,k}(f) \|_{L_p(\mu_j)} &=&  \mu_j\left( |\overline{g}_{j,j+1} \hat{q}_{j+1,k}(K_{k}(f))|^p\right)^{\frac1p}\nonumber\\
&\leq& \gamma^{\frac{p-1}{p}} \mu_{j+1}\left( |\hat{q}_{j+1,k}(K_{k}(f))|^p\right)^{\frac1p}\nonumber\\
&\leq& \gamma^{\frac{p-1}{p}} \delta\, \| K_k(f) \|_{L_q(\mu_k)}\nonumber\\
&\leq& \gamma^{\frac{p-1}{p}} \delta\, \| f \|_{L_q(\mu_k)},\nonumber
\end{eqnarray}
where in the last step we used that by Jensen's inequality $|K_k(f)|^q\leq K_k(|f|^q)$ and that $\mu_k$ is a stationary distribution for $K_k$.
\end{proof}

Combining Theorem \ref{propLpIt} with Corollary \ref{corHypqhat} and Lemma \ref{lemqhatq} we immediately obtain the type of bound needed in the error bounds of Section \ref{secQuantCBG}.\bigskip

\begin{cor}\label{corNEULP}
Consider $p\geq 1$ and $q\geq 1$. Let $q\in [2^r,2^{r+1}]$ for $r\in \mathbb{N}$ and assume $\alpha \gamma^{2^{r+1}-2}<1$. Assume that for all $1\leq j \leq n$ we have a constant $\theta(p,q)\geq 0$ such that
\begin{equation}
\|K_{j}(f)\|_{L_{p} (\mu_{j})}\leq \theta(p,q)\|f\|_{L_{q} (\mu_{j})}. \nonumber
\end{equation}
Then for all $1 \leq j< k \leq n$ and $f\in B(E)$ we have
\[
\|q_{j,k}(f)\|_{L_{p} (\mu_{j})}\leq \widetilde{c}_{j,k}(p,q) \|f\|_{L_{q} (\mu_{k})} 
\]
with
\[
\widetilde{c}_{j,k}(p,q)=\theta(p,q) \gamma^{\frac{p-1}{p}} \gamma^{\frac{q-1}{q}}\delta(q),
\]
where $\delta(q)$ as defined in Corollary \ref{corLp}.
\end{cor}

To round out the analysis of this section, we show that for functions $f$ with $\mu_k(f)=0$ we can moreover show the following result of exponential decay of $\|q_{j,k}(f)\|_{L_{p} (\mu_{j})}$:\bigskip

\begin{prop}\label{expdecay}
Let $p\geq 2$ with $p=2^r$ for $r\in \mathbb{N}$. Assume that $\theta_p=\alpha \gamma^{2p-2}<1$. Then for $1\leq j<k\leq n$ and $f\in B(E)$ with $\mu_k(f)=0$ we have
\[
\|\hat{q}_{j,k}(f)\|^p_{L_{p} (\mu_{j})}\leq \lambda_p\, \theta_p^{k-j} \|f\|^p_{L_{p} (\mu_{k})}, 
\]
where the constants $\lambda_p$ can be calculated recursively from $\lambda_2=1$ and
\[
\lambda_{2p}=1+\lambda_p^2 \, \left(\alpha \left(1-\frac{\alpha}{\gamma^2}\right)\right)^{-1}.
\]
Moreover,
\[
\lambda_p \leq \left[2\,\left(\alpha \left(1-\frac{\alpha}{\gamma^2}\right)\right)^{-1}\right]^{\frac{p}{2}-1}
\]
\end{prop}

\begin{proof}
From (\ref{L2It1}) and $\mu_k(f)=0$ we immediately get the result for $p=2$. Now we proceed inductively concluding from a bound for $p$ a bound for $2p$. Assume thus $\theta_{2p}<1$ and that we have 
\begin{equation}\label{IVexpdec}
\|\hat{q}_{j,k}(f)\|^p_{L_{p} (\mu_{j})}\leq \lambda_p \theta_p^{k-j} \|f\|^p_{L_{p} (\mu_{k})} 
\end{equation}
for $\theta_p $ as defined above and for some $\lambda_p \geq 1$. Observe that $\theta_{2p}<1$ implies immediately $\theta_p<1$. From (\ref{L2pIneqIt}) and (\ref{IVexpdec}) and since we assumed $\lambda_p\geq 1$ we have the inequality
\[
\|\hat{q}_{j,k}(f)\|_{L_{2p} (\mu_{j})}^{2p}\leq \theta_{2p}^{k-j} \|f\|_{L_{2p} (\mu_{k})}^{2p}+\gamma^{2p-2}\sum_{i=j+1}^{k}\theta_{2p}^{i-1-j} \lambda_p^2 \theta_p^{2 (k-i)} \|f\|_{L_{p} (\mu_{k})}^{2p}.
\]
Thus we have
\[
\|\hat{q}_{j,k}(f)\|_{L_{2p} (\mu_{j})}^{2p}\leq \widetilde{\lambda}_{2p}\, \theta_{2p}^{k-j} \|f\|_{L_{2p}(\mu_{k})}^{2p}
\]
with
\[
\widetilde{\lambda}_{2p}=1+\lambda_p^2 \gamma^{2p-2}\theta_{2p}^{-1}\sum_{i=j+1}^{k} \left(\frac{\theta_p^2}{\theta_{2p}} \right)^{k-i}=1+\lambda_p^2 \gamma^{2p-2}\theta_{2p}^{-1}\sum_{i=0}^{k-j-1} \left(\frac{\theta_p^2}{\theta_{2p}} \right)^i.
\]
Observing that 
\[
\gamma^{2p-2} \,\theta_{2p}^{-1} =\frac1\alpha
\]
and
\[
\frac{\theta_p^2}{\theta_{2p}} =\frac{\alpha}{\gamma^2}<1,
\]
and applying the geometric series inequality thus yields
\[
\widetilde{\lambda}_{2p}\leq 1+\lambda_p^2 \, \left(\alpha \left(1-\frac{\alpha}{\gamma^2}\right)\right)^{-1}.
\]
Choosing
\[
\lambda_{2p}=1+\lambda_p^2 \, \left(\alpha \left(1-\frac{\alpha}{\gamma^2}\right)\right)^{-1}
\]
we have thus shown the desired decay inequality. Moreover, observe that, since $\gamma>1$ and $\alpha<1$, $\lambda_p\geq 1$ implies that $\lambda_{2p}\geq 1$ and thus we have $\lambda_p\geq 1$ for all $p=2^r$. To show the upper bound on the coefficients $\lambda_p$, define
\[
\kappa = 2 \left(\alpha \left(1-\frac{\alpha}{\gamma^2}\right)\right)^{-1}.
\]
Since $\kappa > 2$ and $\lambda_p \geq 1$, we have $\lambda_{2p} \leq \kappa \lambda_p^2$. Since $\lambda_2=1$, this implies
\[
\lambda_{2p} \leq \kappa^{p-1}.
\]
\end{proof}

By the Riesz-Thorin interpolation theorem, Proposition \ref{expdecay} immediately generalizes to the case $p \neq 2^r$. Corollary \ref{corHypqhat} and Lemma \ref{lemqhatq} can be used to extend Proposition \ref{expdecay} to $L_p$-$L_q$-bounds and to the $q_{j,k}$.

\section{Error Bounds from Poincaré and Logarithmic Sobolev Inequalities}\label{GLOPOLOSO}
In Section \ref{secStabFKP} we show that assuming (global) Poincaré inequalities associated with sufficiently large spectral gaps for the kernels $K_k$ is sufficient for guaranteeing that the results on $L_p$-Stability of the Feynman-Kac propagator $q_{j,k}$ from Section \ref{secStabFK} can be applied. In Section \ref{secExpErrB} we then give a more explicit version of our error bound for Sequential MCMC in terms of the constants in Poincaré and Logarithmic Sobolev Inequalities. We continue to work in the framework introduced in Section \ref{Lpsetting}.

\subsection{Poincaré Inequalities and Stability of Feynman-Kac Propagators}\label{secStabFKP}

Our $L_p$-stability results of Section \ref{secStabFK} relied on Assumption \ref{ass2}, namely,
\[
\| \hat{q}_{k-1,k}(f) \|_{L_2(\mu_{k-1})}^2\leq \alpha \| f \|_{L_2(\mu_{k})}^2 + \beta \mu_k(f)^2
\]
for coefficients $\alpha>0$ and $\beta \in [0,1]$, where
\[
\hat{q}_{k-1,k}(f)=K_{k-1}\left(\overline{g}_{k-1,k} f \right) \;\;\;\; \forall f \in B(E),
\]
and on the Assumption \ref{ass1} of an upper bound $\gamma$ on the normalized relative densities $\overline{g}_{k-1,k}$. Additionally, we needed conditions assuming that $\alpha$ is sufficiently small.\bigskip

In the following we relate Assumption \ref{ass2} to Poincaré inequalities for the transition kernels $K_k$. We first show that Assumption \ref{ass2} holds, provided that the following $L_2$-inequalities for the kernels $K_k$ are satisfied: 
\begin{ass}\label{assPoinc}
For $1\leq k \leq n$ and $\rho \in (0,1)$ assume that for all $f\in B(E)$, 
\begin{equation}\label{rhobd}
\text{Var}_{\mu_k}(K_k(f)) \leq (1-\rho) \text{Var}_{\mu_k}(f).
\end{equation}
\end{ass}

\begin{lem}\label{lemL2S1}
Assume that Assumption \ref{assPoinc} is satisfied for some $\rho \in (0,1)$ and Assumption \ref{ass1} is satisfied for some $\gamma$. Then Assumption \ref{ass2} holds with
\[
\alpha=(1-\rho) \gamma \text{  and   }\beta = \rho.
\]
\end{lem}

\begin{proof}
Note that we can write
\begin{eqnarray}
\| \hat{q}_{k-1,k}(f) \|_{L_2(\mu_{k-1})}^2&=&\mu_{k-1}(K_{k-1}(\overline{g}_{k-1,k} f)^2)
\nonumber\\
&=&\mu_{k-1}(K_{k-1}(\overline{g}_{k-1,k}f - \mu_{k-1}( \overline{g}_{k-1,k}f  ))^2)+\mu_{k-1}( \overline{g}_{k-1,k}f  )^2.\nonumber
\end{eqnarray}
Thus by (\ref{rhobd}) we have
\begin{eqnarray}
\| \hat{q}_{k-1,k}(f) \|_{L_2(\mu_{k-1})}^2 &\leq& (1-\rho) \left(\mu_{k-1}((\overline{g}_{k-1,k}f)^2)-\mu_{k-1}( \overline{g}_{k-1,k}f  )^2\right)+\mu_{k-1}( \overline{g}_{k-1,k}f  )^2\nonumber\\
&\leq& (1-\rho)\gamma \mu_k(f^2)+\rho \mu_k(f)^2,\nonumber
\end{eqnarray}
which proves the claim.
\end{proof}

We next show how the constant $\rho$ from (\ref{rhobd}) can be controlled in terms of lower bounds on the spectral gaps of the kernels $K_k$. We add one additional assumption for the remainder of Section \ref{secStabFKP}: Assume that $K_k$ is reversible with respect to $\mu_k$, i.e., for all $f, g \in B(E)$
\[
\mu_k(g K_k(f))=\mu_k(f K_k(g)).
\]
Reversibility is, for instance, fulfilled by construction for Metropolis chains.\bigskip

\begin{lem}\label{PoincarLem}
Assume that for all $1 \leq k \leq n$ we have a $\lambda_k\in(0,1)$ such that $K_k$ fulfills a Poincaré inequality with constant $\lambda_k$,
\begin{equation}\label{Poincar}
\lambda_k\, \mu_k(f^2)\leq \mu_k(f \, (I-K_k) (f)) 
\end{equation}
for all $f\in B(E)$ with $\mu(f)=0$, where $I$ denotes the identity mapping on $E$. Then we have
\[
\mu_k(K_k(f-\mu_k(f))^2) \leq (1-\lambda_k)^2 \text{Var}_{\mu_k}(f)
\]
for all $f \in B(E)$. In particular, Assumption \ref{assPoinc} holds with
\[
\rho=\min_{k} \; (1-\lambda_k)^2.
\]
\end{lem}
\begin{proof}
By (\ref{Poincar}) we have for $f\in B(E)$ with $\mu(f)=0$ and $f \not\equiv 0$,
\[
\frac{\mu_k(f\,K_k(f))}{\mu_k(f^2)}\leq 1-\lambda_k,
\]
and thus the spectrum of $K_k$ is bounded from above by $1-\lambda_k$ on $\text{span}(1)^\bot$. Consequently, the spectrum of $K_k^2$ is bounded from above by $(1-\lambda_k)^2$ on $\text{span}(1)^\bot$, i.e.,
\[
\frac{\mu_k(f\,K_k^2(f))}{\mu_k(f^2)}\leq (1-\lambda_k)^2.
\]
By the reversibility of $K_k$, this is equivalent to
\[
\mu_k(K_k(f)^2)\leq (1-\lambda_k)^2 \mu_k(f^2).
\]
To conclude the proof observe that thus for all $f\in B(E)$
\[
\mu_k(K_k(f-\mu(f))^2)\leq (1-\lambda_k)^2 \mu_k((f-\mu_k(f))^2).
\]
\end{proof}

Thus, assuming (\ref{rhobd}) is equivalent to assuming a Poincaré inequality. In the algorithm, $\rho$ can be controlled by varying the number of MCMC steps: A sufficiently large number of MCMC steps makes $\rho$ large and accordingly it makes $\alpha$ small. For future reference, we also give a version of Theorem \ref{propLpIt} under the stronger assumption that (\ref{rhobd}) holds for some sufficiently large $\rho \in (0,1)$. This follows immediately from the Theorem by inserting the values of $\alpha$ and $\beta$ from Lemma \ref{lemL2S1}.
\bigskip

\begin{cor}\label{corLpIt}
Assume that $(\ref{rhobd})$ holds for some $\rho$ with $(1-\rho) \gamma^{p-1}<1$. Consider  $p=2^r$ for $r\in \mathbb{N}$. Then we have for $1\leq j< k \leq n$ and $f\in B(E)$ the inequality
\[
\|\hat{q}_{j,k}(f)\|_{L_{p} (\mu_{j})}\leq \delta(p)\|f\|_{L_{p} (\mu_{k})} 
\]
with
\[
\delta(p)=\prod_{j=1}^r \frac{\gamma^{1-2^{-(j-1)}}}{(1-(1-\rho)\gamma^{2^j-1})^{2^{-j}}}< \frac{\gamma^{r-2+2^{-(r-1)}}}{1-(1-\rho)\gamma^{2^r-1}}.
\]
\end{cor}

\subsection{Explicit Error Bounds}\label{secExpErrB}
In this section we introduce a further parameter $t_k$ for our transition operators $K_k$ which is thought to be the running time or number of MCMC steps contained in $K_k$. We write $K_k^{t_k}$ in the following to make this dependence clear. Throughout this section, we assume two families of mixing conditions depending on $t_k$, a hypercontractivity inequality and an $L_{2}$-$L_2$ inequality
\begin{ass}\label{mixingt}
There exist positive constants $a_k^*$ and $b_k^*$ such that for all $f\in B(E)$ 
\begin{equation}\label{ineqHyper}
\|K_k^{t_k}(f)\|_{L_{q(p,t_k)}(\mu_k)}\leq \|f\|_{L_{p}(\mu_k)},
\end{equation}
where $q(p,t_k)=1+(p-1)\exp(2 a_k^* t_k)$, and 
\begin{equation}\label{ineqL2}
\|K_k^{t_k}(f)-\mu_k(f)\|^2_{L_{2}(\mu_k)}\leq \exp(-2 b_k^* t_k)\; \|f-\mu_k(f)\|^2_{L_{2}(\mu_k)}.
\end{equation}
\end{ass}
Inequalities of this type follow, respectively, from a Logarithmic Sobolev inequality and a Poincaré inequality for the underlying MCMC dynamics, see, e.g., Deuschel and Stroock \cite{DS90} or Ané et al. \cite{ABCFGMRS00}. Furthermore we assume again boundedness of relative densities, i.e., Assumption \ref{ass1}. All proofs are at the end of the section.\bigskip 

We have the following bounds for $q_{j,k}$ provided that the running times $t_j,\ldots,t_k$ are chosen sufficiently large.

\begin{prop}\label{propcpq}
Fix $0 \leq j <k \leq n$, $\tau \in (0,1)$, $1\leq s\in \mathbb{N}$ and $p=2^s$. Assume that for $j \leq l\leq k$
\begin{equation}\label{ineqtl1}
t_l\geq \frac{1}{2 b_l^*}\Big[ (p-1)\log(\gamma)-\log(1-\tau)\Big].
\end{equation}
Then we have for $f \in B(E)$
\[
\| q_{j,k}(f)\|_{L_p(\mu_j)}\leq \widetilde{c}_{j,k}(p,p) \| f \|_{L_p(\mu_k)}
\]
with
\[
\widetilde{c}_{j,k}(p,p) = \frac{\gamma^{s-1+1/p}}{\tau}.
\]
If in addition we have for $p'>p$ and $j \leq l\leq k$,
\begin{equation}\label{ineqtl2}
t_j \geq \frac{1}{2 a_j^*}\Big[ \log(p'-1)-\log(p-1)\Big],
\end{equation}
then we have for $f \in B(E)$
\[
\| q_{j,k}(f)\|_{L_{p'}(\mu_j)}\leq \widetilde{c}_{j,k}(p',p) \| f \|_{L_{p}(\mu_k)}
\]
with
\[
\widetilde{c}_{j,k}(p',p)= \frac{\gamma^{s-1+1/p}}{\tau} \gamma^{\frac{p'-1}{p'}}.
\]
\end{prop}

Here and in the following, it is straightforward to relax the requirement of $p=2^s$, see Corollary \ref{corLp}. The constant $\tau$ controls the contractivity in $L^2$ of the MCMC steps in the following sense: In order to apply our error bounds of Section \ref{secStabFKP}, namely, Corollary \ref{corLpIt}, we need to ensure that
\[
0 < 1- \gamma^{p-1}(1-\rho_l)\stackrel{(\ref{ineqL2})}{=} 1- \gamma^{p-1}e^{-a_l^* t_l}.
\]
$\tau$ is a uniform measure of by how much this inequality is satisfied, i.e., $\tau$ is assumed to be a constant with
\[
\tau < 1- \gamma^{p-1}(1-e^{-a_l^* t_l})\;\;\;\; \text{ for all }\;\; j\leq l \leq k.
\]
Our next step is to utilize the constants $\widetilde{c}_{j,k}(p',p)$ in order to bound the constants that arise in the error bound of Theorem \ref{thmBound}.\bigskip

\begin{cor}\label{corconstants}
Fix $0 \leq j <k \leq n$, $\tau \in (0,1)$, $2\leq s\in \mathbb{N}$ and $p=2^s$. Assume that for $1 \leq l\leq n$,
\begin{equation}\label{ineqtl1b}
t_l\geq \frac{1}{2b_l^*}\Big[ (2p-1)\log(\gamma)-\log(1-\tau)\Big],
\end{equation}
and
\begin{equation}\label{ineqtl2b}
t_l \geq \frac{1}{2 a_j^*}\Big[ \log(p-1)-\log\left(\frac{p}{2}-1\right)\Big].
\end{equation}
Define
\[
h(p) =\frac{\gamma^{2 s+\frac{1}{p}}}{\tau^2}.
\]
Then we have
\[
c_{j,k}(p) \leq h(p).
\]
Furthermore
\[
\widehat{c}_{k}(p)\leq \overline{c}_{k}(p) \leq ((1+\gamma)\vee 3)\, k \,h(p),\;\;\;\;\;\text{ and }\;\;\;\;\; \widehat{v}_{k}(p)\leq
\overline{v}_{k}(p) \leq (k+1) \frac{\gamma}{\tau^2}.
\]
\end{cor}

Adding the final observations that for $f\in B(E)$ and $0\leq j < n$,
\[
\text{Var}_{\mu_j}(q_{j,n}(f)) \leq b_{j,n}(2,2)^2 \|f\|^2_{L_p(\mu_n)}
\]
and
\[
\text{Var}_{\mu_n}(q_{n,n}(f)) \leq \|f\|^2_{L_p(\mu_n)},
\]
we are now in the position to bound all the terms in the error bound of Theorem \ref{thmBound} through $\gamma$, $\tau$, $p$, $N$ and $n$. Thus we arrive at the following version of the corollary:\bigskip

\begin{cor}\label{corErrorExplicit}
Fix $0<n\in \mathbb{N}$, $\tau \in (0,1)$, $2\leq s\in \mathbb{N}$ and $p=2^s$. Assume that for $1 \leq l\leq n$
\[
t_l\geq \frac{1}{2 b_l^*}\Big[ (2p-1)\log(\gamma)-\log(1-\tau)\Big]
\]
and
\[
t_l \geq \frac{1}{2 a_l^*}\Big[ \log(p-1)-\log\left(\frac{p}{2}-1\right)\Big].
\]
Then for $f\in B(E)$ we have
\begin{eqnarray}
&&N \mathbb{E}[|\nu_n^N(f)-\mu_n(f)|^2] \nonumber\\\nonumber\\
&&\leq \|f\|^2_{L_p(\mu_n)}\left[
1+n \gamma \tau^{-2}+n ((1+\gamma)\vee 3)\gamma^{2s+p^{-1}}\tau^{-2}\overline{\varepsilon}^{N,p}_n
\right].\nonumber
\end{eqnarray}
If in addition 
\begin{equation}\label{bdN}
N \geq 2 ((1+\gamma)\vee 3) n \gamma^{2s+p^{-1}}\tau^{-2},
\end{equation}
then we have
\begin{eqnarray}
\overline{\varepsilon}^{N,p}_n \leq
2 \frac{1+n \gamma \tau^{-2}}{N}.
\nonumber
\end{eqnarray}
\end{cor}

Finally, for the sake of illustration we also state these bounds for a concrete choice of parameters, namely $\gamma=2$, $\tau=0.8$, $p=4$ and thus $s=2$. After rounding the coefficients to improve readability (in a way that makes the inequality slightly worse) and inserting the bound on $\overline{\varepsilon}^{N,p}_n$, this yields the bound
\[
\mathbb{E}[|\nu_n^N(f)-\mu_n(f)|^2] \; \leq \; \|f\|^2_{L_4(\mu_n)}\left[
\frac{1+4 n}{N} +
\frac{180 n + 560 n^2}{N^2}
\right].
\]
The required lower bound (\ref{bdN}) on $N$ is given by $N \geq 180$ in this case.

\begin{proof}[Proof of Proposition \ref{propcpq}]
Comparing inequalities (\ref{rhobd}) and (\ref{ineqL2}) shows that the terms $\exp(-2 b_l^* t)$ play the role of $(1-\rho)$ in the setting of Section \ref{secStabFKP}. Thus, assuming for all $j\leq l \leq k$
\[
1-e^{-2b_l^* t_l}\gamma^{p-1}>\tau
\] 
or, equivalently, (\ref{ineqtl1}) ensures that we can apply Corollary \ref{corLpIt} to obtain the bound
\begin{equation}\label{ineqqhatpp}
\| \hat{q}_{j,k}(f)\|_{L_p(\mu_j)}\leq \delta(p) \| f \|_{L_p(\mu_k)}
\end{equation}
with
\[
\delta(p) = \frac{\gamma^{s-2+2/p}}{\tau}. 
\]
Now applying Lemma \ref{lemqhatq} allows to conclude from this bound for $\hat{q}_{j,k}$ the desired $L_p$-$L_p$-bound for $q_{j,k}$ with
\[
\widetilde{c}(p,p) = \frac{\gamma^{s-2+2/p}}{\tau}\, \gamma^{\frac{p-1}{p}}= \frac{\gamma^{s-1+1/p}}{\tau}. 
\]
We next turn to the $L_{p'}$-$L_p$-bound. By (\ref{ineqHyper}), ensuring 
\[
p'\leq 1+(p-1)e^{2 a_j^* t}
\]
or, equivalently, (\ref{ineqtl2}) is a sufficient condition for
\[
\|K_j^{t_j}(f)\|_{L_{p'}(\mu_j)}\leq \|f\|_{L_{p}(\mu_j)}.
\]
Thus we conclude from applying first Corollary \ref{corHypqhat} and then Lemma \ref{lemqhatq} to (\ref{ineqqhatpp}) the desired $L_{p'}$-$L_p$-bound for $q_{j,k}$ with
\[
\widetilde{c}_{j,k}(p',p) = \frac{\gamma^{s-1+1/p}}{\tau}\, \gamma^{\frac{p-1}{p}}\, \gamma^{\frac{p'-1}{p'}} =\frac{\gamma^{s-1+1/p}}{\tau}\, \gamma^{\frac{p'-1}{p'}}.
\]
\end{proof}

\begin{proof}[Proof of Corollary \ref{corconstants}]
Choose $\widetilde{c}_{j,k}(p',p)$ as in Proposition \ref{propcpq}. Since
\[
c_{j,k}(p)=\max\left( 1,\widetilde{c}_{j,k}\left(p,\frac{p}2\right),\widetilde{c}_{j,k}(2,p)^2\right),
\]
by Lemma \ref{cjklemma}, we need to apply Proposition \ref{propcpq} for $(p,p/2)$ and $(2p,p)$. Note that if inequality (\ref{ineqtl1}) holds for $2p$ it also holds for $p$. Conversely, if (\ref{ineqtl2}) holds for $(p,p/2)$ it also holds for $(2p,p)$ since
\[ \frac{2x-1}{x-1}=2+ \frac{1}{x-1}\]
is decreasing in $x$. This motivates our assumption of (\ref{ineqtl1b}) and (\ref{ineqtl2b}). Now observe that
\[
1\leq \widetilde{c}_{j,k}\left(p,\frac{p}2\right)\leq  \widetilde{c}_{j,k}(2p,p)^2= \frac{\gamma^{2s+\frac1p}}{\tau^2}=h(p)
\]
and thus we have
\[
c_{j,k}(p)\leq h(p).
\]
Since for all $x\in E$ we have
\[
-1 \leq q_{j,j+1}(1)(x)-1=\overline{g}_{j,j+1}(x)-1 \leq \gamma-1,
\]
it follows that
\[
2+\|q_{j,j+1}(1)-1\|_{L_p(\mu_j)} \leq ((1+\gamma)\vee 3)
\]
and we can bound $\widehat{c}_k(p)$ as follows:
\[
\widehat{c}_k(p)= \sum_{j=0}^{k-1}c_{j,k}(p) (2+\|q_{j,j+1}(1)(x)-1\|_{L_p(\mu_j)})\leq ((1+\gamma)\vee 3)\, h(p)\, k.
\]
Since this upper bound is monotonically increasing in $k$, it also applies to $\overline{c}_k(p)$. We now turn to $\widehat{v}_k(p)$. Observe that we have
\begin{eqnarray}
\widehat{v}_k(p)&=&\sup\left\{\left. \sum\nolimits_{j=0}^k \text{Var}_{\mu_j}(q_{j,k}(f)) \right| f\in B(E), \|f\|_{L_p(\mu_k)}\leq 1
\right\}\nonumber\\
&=&\sup\left\{\left. \sum\nolimits_{j=0}^k \mu_j(q_{j,k}(f)^2) \right| f\in B(E), \|f\|_{L_2(\mu_k)}\leq 1
\right\}\nonumber\\
&\leq& \sum_{j=0}^k \widetilde{c}_{j,k}(2,2)^2.\nonumber
\end{eqnarray}
As we have
\[
\widetilde{c}_{j,k}(2,2)=\frac{\sqrt{\gamma}}{\tau},
\]
we get the desired upper bound on $\widehat{v}_k(p)$. Since this upper bound is increasing in $k$, it also applies to $\overline{v}_k(p)$.
\end{proof}

\section{Dimension Dependence}\label{secDD}

A particular advantage of our error bounds is that they allow for deriving the algorithm's dimension dependence fairly explicitly for the case where the measures $\mu_k$ are product measures on $\mathbb{R}^d$. This can be seen as a first step to understanding the algorithm's overall dimension dependence.\bigskip 

We demonstrate that if we have a one-dimensional setting where our bounds apply, then we can obtain bounds of the same order for the $d$-dimensional product of the one-dimensional target distribution by increasing the computational effort by a factor of order $O(d^3)$ and thus by a factor which is polynomial in $d$. Consider a sequence of distributions $\mu_0$,$\ldots$, $\mu_n$ on $\mathbb{R}$ such that relative densities are bounded by $\gamma$ and such that we can obtain sufficiently good mixing properties for the dynamics $K_k$ from Logarithmic Sobolev inequalities. The constants in Logarithmic Sobolev inequalities are not dimension-dependent for product measures, see Ané et al. \cite{ABCFGMRS00}. Thus, we obtain the same constants in the mixing conditions for the $d$-dimensional product dynamics $K_k^{(d)}$ with target $\mu_k^{\otimes d}$ as for the one-dimensional dynamics $K_k$. Since the $d$-dimensional relative densities $\overline{g}_{k,k+1}^{(d)}$ are $d$-fold products of the one-dimensional densities $\overline{g}_{k,k+1}$, we need to increase the number of interpolating distributions by a factor $d$ when switching from $\mathbb{R}$ to $\mathbb{R}^d$. This can be done by inserting $d-1$ additional distributions $\mu_{k,1}^{\otimes d}$ to $\mu_{k,d-1}^{\otimes d}$ between $\mu_k^{\otimes d}$  and $\mu_{k+1}^{\otimes d}$ where, for $1 \leq j < d$, $\mu_{k,j}^{\otimes d}$ is given by
\[
\mu_{k,j}^{\otimes d}(dx_1,\ldots, dx_d) = \left( \prod_{l=1}^{d} \overline{g}_{k,k+1}(x_l)
\right)^{j\,d^{-1}} \mu_{k}^{\otimes d}(dx_1,\ldots, dx_d).
\]
Note that $\mu_{k,0}^{\otimes d}=\mu_{k}^{\otimes d}$ and $\mu_{k,d}^{\otimes d}=\mu_{k+1}^{\otimes d}$. Moreover, the relative densities between $\mu_{k,j}^{\otimes d}$ and $\mu_{k,j+1}^{\otimes d}$ are bounded by $\gamma$.  Assume furthermore that the Logarithmic Sobolev constant for the $d-1$ interpolating measures ``inserted'' between $\mu_k$ and $\mu_{k+1}$ lies in between the Logarithmic Sobolev constants for $\mu_k^{(d)}$ and $\mu_{k+1}^{(d)}$. Thus MCMC with respect to the inserted distributions is not more difficult than MCMC with respect to the original distributions. Now re-index the sequence of $n^{(d)}=n \cdot d +1$ product measures on $\mathbb{R}^d$ to $\mu_0^{(d)},\ldots \mu_{n_d}^{(d)}$ and denote the associated transition kernels and relative densities by $K_k^{(d)}$ and $\overline{g}_{k,k+1}^{(d)}$. Then, we obtain dimension-independent constants $c_{j,k}(p)$ in the inequalities (\ref{cjninequality}) for the propagator $q_{j,k}^{(d)}$ derived from $K_k^{(d)}$ and $\overline{g}_{k,k+1}^{(d)}$. (Recall that our bounds on the constants $c_{j,k}$ from Section \ref{LPglobal} were independent of $j$ and $k$.) Notably, these constants coincide with those derived for the original one-dimensional sequence $\mu_0,\ldots,\mu_n$ and we have $\mu_n^{\otimes d}=\mu_{n_d}^{(d)}$.\bigskip 

Overall, we then find that the algorithm's error depends on the dimension like $O(d^3)$ in this example: We need one factor $d$ since each step of the product dynamics is more costly to simulate, one factor $d$ due to the fact that we increase the number of levels by a factor $d$, and one factor $d$ because we have to adjust the number of particles $N$ by the same factor as the number of levels $n$ as can be seen from Theorem \ref{thmBound}. It is an open question whether similar results can be derived from minorization and drift conditions as in Whiteley \cite{W11}.


\begin{thebibliography}{99}\bibitem{ABCFGMRS00} C. Ané, S. Blachère, D. Chafaï, P. Fougères, I. Gentil, F. Malrieu, C. Roberto and G. Scheffer, \textit{Sur les inégalités de Sobolev logarithmiques}, Panoramas et
Synthèses, 10, Société Mathématique de France, Paris, 2000.

\bibitem{CMR05} O. Cappé, E. Moulines and T. Rydén, \textit{Inference in Hidden Markov Models}, Springer, New York, 2005.

\bibitem{CDG11}
F. C\'erou, P. Del Moral and A. Guyader,\textit{ A nonasymptotic variance theorem for unnormalized Feynman-Kac particle models}, Annales de l'Institut Henri Poincar\'e, Probabilit\'es et Statistiques, 47, 629-649, 2011. 


\bibitem{C04} N. Chopin, \textit{Central Limit Theorem for Sequential Monte Carlo methods and its application to Bayesian inference}, Annals of Statistics, 32, 2385-2411, 2004.


\bibitem{DA90} E. B. Davies, \textit{Heat kernels and spectral theory}, Cambridge University Press, Cambridge, 1990.


\bibitem{DM96}
P. Del Moral, \textit{Nonlinear filtering: interacting particle solution}, Markov Processes and Related Fields, 2, 555-579, 1996.

\bibitem{DM05}
P. Del Moral, \textit{Feynman-Kac Formulae}, Springer, New York, 2004.

\bibitem{DDJ05} P. Del Moral, A. Doucet and A. Jasra, \textit{Sequential Monte Carlo Samplers}, Journal of the Royal Statistical Society B, 68, 411-436, 2006.


\bibitem{DM00} P. Del Moral and L. Miclo, \textit{Branching and interacting particle systems approximations of Feynman-Kac formulae with applications to nonlinear filtering}, Séminaire de Probabilités XXXIV, Lecture Notes in Mathematics, vol. 1729, Springer, Berlin, 2000, pp. 1-145. 

\bibitem{DS90} J.-D. Deuschel and D. W. Stroock, \textit{Hypercontractivity and spectral gap of symmetric diffusions with applications to the stochastic Ising models}, Journal of Functional Analysis, 92, 30-48, 1990.


\bibitem{DM08} R. Douc and E. Moulines, \textit{Limit theorems for weighted samples with applications to Sequential Monte Carlo Methods}, Annals of Statistics, 36, 2344-2376, 2008.

\bibitem{DMR04}
R. Douc, E. Moulines and J. S. Rosenthal, \textit{Quantitative bounds on convergence of time-inhomogeneous Markov chains}, Annals of Applied Probability, 14, 1643-1665, 2004.

\bibitem{EM09} A. Eberle and C. Marinelli, \textit{$L^p$ estimates for Feynman-Kac propagators with time-dependent reference measures}, Journal of Mathematical Analysis and Applications, 365, 120-134,  2010.

\bibitem{EM08} A. Eberle and C. Marinelli, \textit{Quantitative approximations of evolving probability measures and sequential Markov Chain Monte Carlo methods}, Probability Theory and Related Fields, forthcoming, 2012.




\bibitem{GSS93} N. Gordon, D. Salmond and A. Smith, \textit{Novel approach to nonlinear/non-Gaussian Bayesian state estimation}, IEE Proceedings F Radar and Signal Processing, 140, 107-113, 1993.


\bibitem{K05}
H. R. K\"unsch, \textit{Recursive Monte-Carlo filters: algorithms and theoretical analysis}, Annals of Statistics, 33, 1983-2021, 2005.


\bibitem{Liu01}
J. Liu, \textit{Monte Carlo strategies in scientific computing}, Springer, New York, 2001.


\bibitem{MRRTT53}
N. Metropolis, A. W. Rosenbluth, M. N. Rosenbluth, A. H. Teller and E. Teller, \textit{Equations of state calculations by fast computing machines}, Journal of Chemical Physics, 21, 1087-1092, 1953.

\bibitem{N01} R. M. Neal, \textit{Annealed importance sampling}, Statistics and Computing, 11, 125-139, 2001.


\bibitem{Schw11}
N. Schweizer, \textit{Non-asymptotic Error Bounds for Sequential MCMC Methods.}, Doctoral Thesis, University of Bonn, 2011.

\bibitem{Schw12}
N. Schweizer, \textit{Non-asymptotic Error Bounds for Sequential MCMC Methods in Multimodal Settings.}, in preparation.

\bibitem{W11}
N. Whiteley, \textit{Sequential Monte Carlo samplers: error bounds and insensitivity to initial conditions}, Working Paper, University of Bristol, 2011.


\end{thebibliography}
\end{document}